\documentclass[10pt,reqno,oneside,a4paper]{amsart}

\usepackage{subfiles}
\usepackage[english]{babel}
\usepackage[utf8]{inputenc}
\usepackage{amsmath}
\usepackage{amsthm}
\usepackage{amssymb}
\usepackage{amsfonts}
\usepackage{graphicx}
\usepackage[shortlabels]{enumitem}

\usepackage[top=30truemm,bottom=30truemm,left=30truemm,right=30truemm]{geometry}

\usepackage{adjustbox}
\usepackage{listings}
\usepackage{color}
\usepackage{upgreek}
\usepackage{mathtools}
\usepackage{color}
\usepackage{empheq}
\usepackage{bm}
\usepackage{csquotes}
\usepackage{tikz}
\usetikzlibrary{matrix}
\usepackage{tikz-cd}
\usepackage{setspace}
\usepackage[norelsize]{algorithm2e}
\usepackage{adjustbox}
\usepackage{mathrsfs}
\usepackage{stmaryrd}
\usetikzlibrary{decorations.pathmorphing,shapes}
\usetikzlibrary{decorations.pathreplacing,calligraphy}
\usepackage{rotating}
\usepackage{pdflscape}
\usepackage{subcaption}
\usepackage{afterpage}
\usepackage[all,cmtip]{xy}
\usepackage{tabu}
\usepackage{stackengine}
\usepackage{caption}
\usetikzlibrary{arrows,calc}
\usepackage{hyperref}
\usepackage{scalerel}
\usepackage{centernot}
\usepackage{float}
\usepackage{tabularx}
\usepackage{empheq}
\usepackage{multirow}
\usepackage{makecell}
\usetikzlibrary{calc}
\usetikzlibrary{arrows.meta}
\usetikzlibrary{chains}

\makeatletter
\DeclareRobustCommand{\rvdots}{%
  \vbox{
    \baselineskip4\p@\lineskiplimit\z@
    \kern-\p@
    \hbox{.}\hbox{.}\hbox{.}
  }}
\makeatother

\allowdisplaybreaks

\newcommand{\tikzAngleOfLine}{\tikz@AngleOfLine}
\def\tikz@AngleOfLine(#1)(#2)#3{%
\pgfmathanglebetweenpoints{%
\pgfpointanchor{#1}{center}}{%
\pgfpointanchor{#2}{center}}
\pgfmathsetmacro{#3}{\pgfmathresult}%
}

\makeatletter
\patchcmd{\@setaddresses}{\indent}{\noindent}{}{}
\patchcmd{\@setaddresses}{\indent}{\noindent}{}{}
\patchcmd{\@setaddresses}{\indent}{\noindent}{}{}
\patchcmd{\@setaddresses}{\indent}{\noindent}{}{}
\makeatother

\newcommand{\gap}{\hspace{1pt}}
\newcommand{\ZZ}{\mathbb{Z}}

\newcommand{\NN}{\mathbb{N}}

\DeclareMathOperator{\GKdim}{GKdim}

\DeclareMathOperator{\gldim}{gl.\hspace{-1pt}dim}

\DeclareMathOperator{\idim}{i.\hspace{-1pt}dim}

\DeclareMathOperator{\End}{End}

\DeclareMathOperator{\Hom}{Hom}

\DeclareMathOperator{\Ext}{Ext}

\DeclareMathOperator{\SL}{SL}

\newcommand{\id}{\text{\normalfont id}}

\newcommand{\hash}{\hspace{1pt} \# \hspace{1pt}}

\numberwithin{equation}{section}

\theoremstyle{definition}

\newtheorem{defn}[equation]{Definition}

\newtheorem{example}[equation]{Example}

\theoremstyle{plain}
\newtheorem{thm}[equation]{Theorem}
\newtheorem{prop}[equation]{Proposition}
\newtheorem{lem}[equation]{Lemma}

\theoremstyle{remark}
\newtheorem{rem}[equation]{Remark}

\newcommand\restr[2]{{% we make the whole thing an ordinary symbol
  \left.\kern-\nulldelimiterspace % automatically resize the bar with \right
  #1 % the function
  \right|_{#2} % this is the delimiter
  }}

\newcounter{sarrow}

\newcounter{darrow}

\makeatletter
\renewcommand*\env@matrix[1][\arraystretch]{%
  \edef\arraystretch{#1}%
  \hskip -\arraycolsep
  \let\@ifnextchar\new@ifnextchar
  \array{*\c@MaxMatrixCols c}}
\makeatother

\newboolean{printBibInSubfiles}
\setboolean{printBibInSubfiles}{true}
\def\bib{\ifthenelse{\boolean{printBibInSubfiles}}
           { \bibliographystyle{amsalpha} \bibliography{thesisbib} }
       {}
}

\makeatletter
\tikzset{
  column sep/.code=\def\pgfmatrixcolumnsep{\pgf@matrix@xscale*(#1)},
  row sep/.code   =\def\pgfmatrixrowsep{\pgf@matrix@yscale*(#1)},
  matrix xscale/.code=%
    \pgfmathsetmacro\pgf@matrix@xscale{\pgf@matrix@xscale*(#1)},
  matrix yscale/.code=%
    \pgfmathsetmacro\pgf@matrix@yscale{\pgf@matrix@yscale*(#1)},
  matrix scale/.style={/tikz/matrix xscale={#1},/tikz/matrix yscale={#1}}}
\def\pgf@matrix@xscale{1}
\def\pgf@matrix@yscale{1}
\makeatother

\definecolor{mygray}{gray}{0.85}

\def\dotfill#1{\cleaders\hbox to #1{.}\hfill}
\makeatletter
\def\myrulefill{\leavevmode\leaders\hrule height .7ex width 1ex depth -0.6ex\hfill\kern\z@}
\makeatother

\title{\vspace*{-30pt}Group coactions on two-dimensional Artin-Schelter regular algebras}
\author{Simon Crawford}
\address{The University of Manchester, Alan Turing Building, Oxford Road, Manchester, M13 9PL, United Kingdom}
\email{simon.crawford@manchester.ac.uk}

\date{\today}

\begin{document}
\begin{abstract}
We describe all possible coactions of finite groups (equivalently, all group gradings) on two-dimensional Artin-Schelter regular algebras. We give necessary and sufficient conditions for the associated Auslander map to be an isomorphism, and determine precisely when the invariant ring for the coaction is Artin-Schelter regular. The proofs of our results are combinatorial and exploit the structure of the McKay quiver associated to the coaction.
\end{abstract}
\maketitle

\vspace{-10pt}
\section{Introduction}
Noncommutative invariant theory has been the subject of intense and sustained research during the past two decades, and many results from commutative invariant theory have been shown to have noncommutative analogues. The survey article by Kirkman \cite{kirkmansurvey} provides an excellent overview of the area. \\
\indent In the commutative setting, one typically considers the action of a finite group $G$ consisting of graded automorphisms of a polynomial ring $R$, and then studies various rings, modules, and geometric objects that arise from this setup. A natural first step in generalising this framework is to replace the polynomial ring $R$ by a noncommutative analogue, a role which is often played by Artin-Schelter (AS) regular algebras. One shortcoming of this approach is that noncommutative rings often have quite small automorphism groups; however, it is possible to circumvent this issue by also allowing for actions of finite-dimensional Hopf algebras, which are often thought of as encoding \emph{quantum symmetry}. \\
\indent Coactions of finite groups on algebras provide one of the more easily-understood Hopf algebra actions, since they are equivalent to equipping the algebra with a group grading. Accordingly, they often serve as a good test case when attempting to generalise results from commutative invariant theory to noncommutative rings. In spite of this, there does not appear to have been a systematic study of group coactions on two-dimensional AS regular algebras, which is the focus of the present article. In particular, we will study two questions in this context which are motivated by two classical results from commutative invariant theory. \\
%\indent Indeed, these generalisations have proved quite fruitful, with many classical results from commutative invariant theory having noncommutative counterparts. For example, a result of Watanabe \cite[Theorem 1]{watanabe} says that if every element of $G \leqslant \GL(2,\Bbbk)$ has trivial determinant, then $R^G$ is Gorenstein. One of the first major results in noncommutative invariant theory was a generalisation of Watanabe's Theorem: J{\o}rgensen--Zhang proved that, if $A$ is AS regular and $G \leqslant \Autgr(A)$ is a finite group acting on $A$ with \emph{trivial homological determinant}, then $A^G$ is Gorenstein \cite{jorgensen}. This result was later extended to actions of Hopf algebras in \cite{kkzGor}. \textcolor{red}{Maybe get rid of this paragraph} \\
\indent There is a certain ring homomorphism, called the \emph{Auslander map}, associated to the data of an AS regular algebra and an action from a Hopf algebra. If this map is an isomorphism, then there are strong representation-theoretic consequences, and so one of the main research directions in noncommutative invariant theory has been to provide conditions for when this is the case. In the commutative setting of a finite group $G$ acting on a polynomial ring $R$, it is known that the Auslander map is an isomorphism if and only if $G$ contains no non-trivial reflections \cite{auslanderrational}. In particular, if $G \leqslant \SL(n,\Bbbk)$, then the Auslander map is an isomorphism. This result forms part of the celebrated \emph{McKay correspondence}. \\
\indent Versions of this result have been generalised to the noncommutative setting. In \cite{ckwzi}, it was shown that if $A$ is two-dimensional AS regular and $H$ is a Hopf algebra which acts with \emph{trivial homological determinant}, then the Auslander map is an isomorphism. The same result holds $H = \Bbbk G$ for a group $G$ which contains no \emph{quasi-reflections} \cite{crawfordActions}. One of the main results of this paper gives a complete answer for coactions of finite groups of two-dimensional AS regular algebras. As a first step, we give a classification of such coactions, where undefined terms will be defined in Section \ref{sec:Preliminaries}:

\begin{thm}[Proposition \ref{prop:Classifying}] \label{thm:intro1}
Suppose that $A$ is a two-dimensional AS regular algebra which is not commutative and which has an inner-faithful $G$-grading, where $G$ is finite and non-abelian. Then $A \cong \Bbbk\langle u,v \rangle/\langle u^2-v^2 \rangle$ and $G = \langle a,b \mid \mathcal{R} \rangle$ for some set of relations $\mathcal{R}$, one of which is $a^2 = b^2$, and where the grading satisfies $\deg_G u = a$, $\deg_G v = b$.
\end{thm}

While this result imposes quite strong restrictions on the possibilities for $A$ and $G$, one can still obtain many interesting examples which exhibit different representation-theoretic behaviour. In spite of this wealth of examples, we are able to establish a simple, uniform necessary and sufficient condition for the Auslander map to be an isomorphism:

\begin{thm}[{Theorem \ref{AuslanderForCoaction}}] \label{thm:intro2}
With the notation from Theorem \ref{thm:intro1}, the Auslander map associated to a pair $(A,G)$ is an isomorphism if and only if $|ab| = \frac{1}{2} |G|$.
\end{thm}
 
\indent A closely related problem is to identify when the invariant ring $A^H$ associated to this setup is again AS regular. For polynomial rings, a complete answer to this question is given by the Chevalley--Shephard--Todd Theorem, which says that the invariant ring is again a polynomial ring (so, in particular, AS regular) if and only if $G$ is generated by reflections \cite{chevalley,shephardtodd}. Returning to the noncommutative setting, much less is known, although there are families of examples for which we have complete answers. For example, \cite{three} provides many instances of semisimple Hopf algebras $H$ acting on two- and three-dimensional AS regular algebras such that $A^H$ is AS regular, and \cite{chenRigidity} shows that invariant rings of down-up algebras under group coactions are never AS regular. \\
\indent Our second main result gives a complete characterisation of when the invariant rings arising from group coactions are AS regular:

\begin{thm}[{Theorem \ref{thm:InvariantRegular}}] \label{thm:intro3}
The invariant ring associated to a pair $(A,G)$ from Theorem \ref{thm:intro1} is AS regular if and only if 
\begin{align*}
G = \langle a,b \mid a^2=b^2, a^{4m} = b^{4m} = (ab)^m =(ba)^m = 1_G \rangle
\end{align*}
for some $m \geqslant 2$.
\end{thm}

To prove Theorems \ref{thm:intro2} and \ref{thm:intro3}, we first express the invariant ring (and a related ring) as the path algebra of a quiver with relations. This allows us to exploit quiver combinatorics which, when combined with results in the literature, establish our claims. In principle, it should be possible to prove these results without using quivers, but our techniques help illuminate why one might expect these results to be true, and can be generalised to other settings.  \\

\noindent\textbf{Acknowledgements.} The author is a Heilbronn fellow at the University of Manchester. Portions of this work were completed at the University of Washington while the author was in receipt of the Cecil King Travel Scholarship. The author is grateful for their financial support.

\section{Preliminaries} \label{sec:Preliminaries}
Throughout, $\Bbbk$ will denote an algebraically closed field of characteristic $0$. Let $A$ be a $\Bbbk$-algebra. We say that $A$ is \emph{connected graded} if it is $\NN$-graded with $A_0 = \Bbbk$. A graded $A$-module $M$ is said to be \emph{locally finite} if $\dim_\Bbbk M_n < \infty$ for all $n \in \ZZ$. \\
\indent Now assume that $A$ is connected graded and locally finite. The \emph{Gelfand-Kirillov (GK) dimension} of a locally finite graded $A$-module is
\begin{align*}
\GKdim (M) \coloneqq \limsup_{n \to \infty} \log_n(\dim_\Bbbk M_n).
\end{align*}
It is possible to define the GK dimension of a module in wider generality, but this definition is equivalent to the usual one by \cite[Proposition 6.6]{lenagan}. The GK dimension provides a sensible notion of dimension for noncommutative $\Bbbk$-algebras; in particular, if $A$ is commutative, then the GK dimension of $A$ coincides with the Krull dimension of $A$. \\
\indent We will be interested in a particular family of connected graded rings, which may be thought of as noncommutative analogues of polynomial rings:

\begin{defn} \label{ASregdef}
Let $A$ be a connected graded $\Bbbk$-algebra and write $\Bbbk = A/A_{\geqslant 1}$ for the trivial module. We say that $A$ is \emph{Artin--Schelter Gorenstein} (or AS Gorenstein) \emph{of dimension $d$} if:
\begin{enumerate}[{\normalfont (1)},topsep=1pt,itemsep=1pt,leftmargin=35pt]
\item $\idim {}_A A = \idim A_A = d < \infty$, and 
\item There is an isomorphism of graded right $A$-modules
\begin{align*}
\Ext^i(_{A} \Bbbk,{}_A A) \cong \left \{
\begin{array}{cl}
0 & \text{if } i \neq d \\
\Bbbk[\ell]_A & \text{if } i = d
\end{array}
\right .
\end{align*}
for some integer $\ell$, and a symmetric condition holds for $\Ext^i(\Bbbk_{A},A_A)$, with the same integer $\ell$. We call $\ell$ the \emph{Gorenstein parameter} of $A$.
\end{enumerate}
If, moreover
\begin{enumerate}[{\normalfont (1)},topsep=1pt,itemsep=1pt,leftmargin=35pt]
\item[(3)] $\gldim A = d$, and
\item[(4)] $A$ has finite GK dimension,
\end{enumerate}
then we say that $A$ is \emph{Artin--Schelter regular} (or AS regular) \emph{of dimension $d$}.
\end{defn}

As mentioned previously, AS regular algebras are often thought of as noncommutative analogues of polynomial rings, and have many good properties in common with them. For example, all known examples are noetherian domains, and it is conjectured that this is always the case. In particular, commutative AS regular algebras are polynomial rings \cite[Exercise 2.2.25]{bruns}. \\
\indent AS regular algebras are classified in dimensions 2 and 3 \cite{artinschelter,atv,stephenson1,stephenson2}. It is easy to state the classification in dimension 2 (where we will additionally impose the restriction that $A$ is generated in degree $1$). In this case, up to isomorphism, the only examples are
\begin{align} \label{eqn:TwoDimASReg}
\Bbbk_q[u,v] \coloneqq \frac{\Bbbk \langle u,v \rangle}{\langle vu - quv \rangle},\quad q \in \Bbbk^\times, \qquad \text{and} \qquad \Bbbk_J[u,v] \coloneqq \frac{\Bbbk \langle u,v \rangle}{\langle vu - uv - u^2 \rangle},
\end{align}
which are called the \emph{quantum plane} and \emph{Jordan plane}, respectively. We also note that the algebra
\begin{align*}
\frac{\Bbbk \langle x,y \rangle}{\langle x^2-\alpha y^2 \rangle}, \quad \alpha \in \Bbbk^\times,
\end{align*}
is isomorphic to $\Bbbk_{-1}[u,v]$ via the map $u \mapsto x + \sqrt{\alpha} y, \hspace{2pt} v \mapsto x - \sqrt{\alpha} y$, and hence is AS regular. \\
\indent In this paper, we will be interested in actions of a specific family of Hopf algebras (namely, group coactions) on AS regular algebras. We now state a number of definitions and results for general Hopf algebras and their actions, and give specific details for group coactions when appropriate. \\
\indent Given a Hopf algebra $H$, we write $\Delta$ for the coproduct, $\varepsilon$ for the counit, and $S$ for the antipode. We write elements in the image of the coproduct using Sweedler notation so, for $h \in H$,
\begin{align*}
\Delta(h) = \sum h_1 \otimes h_2.
\end{align*}
When $H$ is semisimple, as will be the case for us, it is necessarily finite-dimensional and the antipode satisfies $S^2 = \id_H$. \\
\indent If $H$ is a finite-dimensional Hopf algebra, then the dual algebra $H^* \coloneqq \Hom_\Bbbk(H,\Bbbk)$ has the structure of a Hopf algebra as well. We will be particularly interested in this construction when $H = \Bbbk G$ is the group algebra of a finite group $G$. In this case, $H^*$ has a $\Bbbk$-basis $\{ \phi_g \mid g \in G \}$, where $\phi_g : \Bbbk G \to G$ is dual to the element $g \in G$, and these elements form a set of primitive orthogonal idempotents for $H^*$. The coproduct, counit, and antipode for $H^*$ are defined as follows:
\begin{align*}
\Delta(\phi_g) = \sum_{\substack{x,y \in G \\ xy=g}} \phi_x \otimes \phi_y, \qquad \varepsilon(\phi_g) = \left \{ 
\begin{array}{cc}
1 & \text{if } g = \id_G, \\
0 & \text{otherwise},
\end{array}\right.
\qquad S(\phi_g) = \phi_{g^{-1}}.
\end{align*}
\indent In particular, we shall be interested in actions of Hopf algebras $H$ on $\Bbbk$-algebras $A$ which are compatible with the algebra structure. In this case, there is well-developed invariant theory for the action of $H$ on $A$.

\begin{defn}
Let $H$ be a Hopf algebra and $A$ be a $\Bbbk$-algebra. We say that \emph{$A$ is an $H$-module algebra} if $A$ is a left $H$-module which satisfies
\begin{align*}
h \cdot (ab) = \sum (h_1 \cdot a)(h_2 \cdot b) \qquad \text{and} \qquad h \cdot 1_A = \varepsilon(h) 1_A
\end{align*}
for all $h \in H$ and $a,b \in A$. In this case, the \emph{invariant ring} of the action of $H$ on $A$ is
\begin{align*}
A^H = \{ a \in A \mid h \cdot a = \varepsilon(h) a \text{ for all } h \in H \}.
\end{align*}
When $A$ is an $H$-module algebra, we can form the \emph{smash product} $A \hash H$. As an abelian group, this is simply $A \otimes_\Bbbk H$, and we write a simple tensor as $a \hash h$ where $a \in A$ and $h \in H$. The multiplication is given by 
\begin{align*}
(a \hash h)(b \hash k) = \sum a (h_1 \cdot b) \hash h_2 k,
\end{align*}
and extended linearly.
\end{defn}

We say that an algebra $A$ is \emph{graded by the group $G$}, or \emph{$G$-graded}, if there is a decomposition into abelian groups
\begin{align*}
A = \bigoplus_{g \in G} A_g
\end{align*}
which satisfies $A_g A_h \subseteq A_{gh}$. Necessarily $1_A \in A_{1_G}$. If $A$ is $G$-graded, a (left) $A$-module $M$ is \emph{$G$-graded} if there is a decomposition into abelian groups
\begin{align*}
M = \bigoplus_{g \in G} M_g
\end{align*}
which satisfies $A_g M_h \subseteq M_{gh}$. \\
\indent When $H = (\Bbbk G)^*$, it is well-known that an abelian group $M$ being an $H$-module is equivalent to $M$ being $G$-graded \cite[Example 1.6.7]{montgomerybook}), which is in turn equivalent to defining a coaction of $\Bbbk G$ on $M$. More precisely, given a $G$-grading on $M$, we obtain an action of $(\Bbbk G)^*$ by defining $\phi_g \cdot m \coloneqq m_g$, where $m = \sum_{g \in G} m_g$ is the decomposition of $m \in M$ into its $G$-homogeneous components. Similarly, we can use an action of $(\Bbbk G)^*$ on $M$ to define a $G$-grading on $M$. This correspondence also extends to $H$-module algebras. We will use the terms $(\Bbbk G)^*$-action, $G$-grading, and $G$-coaction interchangeably, and will swap between these different perspectives when convenient. We also note that, in this setting, the invariant ring satisfies $A^H = A_{1_G}$, i.e. it is the set elements of $A$ which have trivial $G$-degree. \\
\indent One of the most important properties of Hopf algebras is that the tensor product of $H$-modules is again an $H$-module. If $H = (\Bbbk G)^*$ and $V$ and $W$ are $H$-modules (i.e. $G$-graded), then so too is $V \otimes_\Bbbk W$, and if $v \in V$ and $w \in W$ are $G$-homogeneous, then
\begin{align*}
\deg_G (v \otimes w) = (\deg_G v)(\deg_G w).
\end{align*}
%When $H = (\Bbbk G)^*$, it is well-known that an algebra $A$ is an $H$-module algebra if and only if it is $G$-graded \cite[Example 1.6.7]{montgomerybook}). More precisely, given a $G$-grading, we obtain an action of $(\Bbbk G)^*$ by defining $\phi_g \cdot a \coloneqq a_g$, where $a = \sum_{g \in G} a_g$ is the decomposition of $a \in A$ into its $G$-homogeneous components. Similarly, we can use an action of $(\Bbbk G)^*$ on $A$ to define a $G$-grading on $A$. We will use this correspondence freely throughout the remainder of this paper. \\
\indent When working with actions of finite groups on $\Bbbk$-algebras, it is typical to assume that the corresponding representation of the group is faithful. The appropriate analogue of a faithful action for Hopf algebras is the following:

\begin{defn}
Let $V$ be a left $H$-module. We say that the action of $H$ on $V$ is \emph{inner-faithful} if $IV \neq 0$ for every nonzero Hopf ideal $I$ of $H$.
\end{defn}

This says that the action of $H$ does not factor through the action of one of its proper quotients. In particular, the action of a group algebra is inner-faithful if and only if it is faithful. \\
\indent Suppose that $A$ is an $\NN$-graded $H$-module algebra and $U$ is a $H$-submodule of $A$ which generates $A$ as an algebra. Then it is easy to see that the action of $H$ on $A$ is inner-faithful if and only if the action of $H$ on $U$ is inner-faithful. In particular, if that action of $H$ is homogeneous and $A$ is generated in degree $1$, then this happens if and only if the action of $H$ on $A_1$ is inner-faithful. \\
\indent When $H = (\Bbbk G)^*$ is the dual of a group algebra, it is easy to detect when the action is inner-faithful:

\begin{lem} \label{lem:innerfaithful}
Suppose that $A$ is a graded $\Bbbk$-algebra which is generated in degree $1$, $H = (\Bbbk G)^*$ for some finite group $G$, and $A$ is a $H$-module algebra. Then the action of $H$ on $A$ is inner-faithful if and only if
\begin{align*}
D \coloneqq\{ \deg_G x \mid x \in A_1 \text{ is } G\text{-homogeneous} \}
\end{align*}
is a generating set for $G$.
\end{lem}
\begin{proof}
The action on $A$ is inner-faithful if and only if the action on $A_1$ is inner-faithful which, by \cite[Corollary 10]{passmanburnside}, happens if and only each irreducible representation $\Bbbk \phi_g$ of $H$ is a summand of some $(A_1)^{\otimes_{\Bbbk} n_g}$ for some $n_g \in \NN$. Equivalently, for each $g \in G$, the tensor algebra $T_\Bbbk(A_1)$ contains an element of degree $g$, which happens if and only if $D$ is a generating set for $G$. 
\end{proof}

% \indent Let $R = \Bbbk[x_1, \dots, x_n]$. In commutative invariant theory, if $G \leqslant \Autgr(R) = \GL_n(\Bbbk)$, then the determinants of the elements of $G$ control many homological properties of the invariant ring $R^G$. For example, a result of Watanabe \cite[Theorem 1]{watanabe} says that if $G \leqslant \SL_n(\Bbbk)$, then $R^G$ is Gorenstein. Here, one may view the determinant as a map
% \begin{align*}
% \det : \Bbbk G \to \Bbbk,
% \end{align*}
% and then Watanabe's Theorem says that, if $\det = \varepsilon$, the counit of $\Bbbk G$, then $R^G$ is Gorenstein. \\
% \indent The above theory has been 
% An important quantity in noncommutative invariant theory is the homological determinant. This plays the same role as th
The so-called \emph{Auslander map} plays an important role in noncommutative invariant theory. This is the natural map
\begin{align*}
\gamma : A \hash H \to \End(A_{A^H}), \quad \gamma(a \hash h)(b) = a (h \cdot b).
\end{align*}
If this map is an isomorphism, then there are strong representation-theoretic results relating the algebras $H$, $A \hash H$, and $A^H$ \cite{ckwzii}, and so there is a considerable interest in determining when this is the case. By a recent result of Zhu, \cite[Theorem 0.3]{zhu}, it is known that this happens when $A$ is an AS regular algebra which satisfies a polynomial identity, and $H$ is a group algebra or the dual of one such that the action of $H$ on $A$ has \emph{trivial homological determinant}. We neglect to provide a definition of this term, but remark that this can be viewed as an analogue of requiring a group $G$ acting on a polynomial ring to satisfy $G \leqslant \SL(n,\Bbbk)$. The interested reader should consult \cite{jorgensen,kkzGor} for more information. \\ 
\indent Many of the results in this paper will be proved using quivers, so we recall some definitions.  A \emph{quiver} is a directed multigraph, and in this paper they will be assumed to have finitely many vertices and edges. For these definitions, we shall assume $Q$ has $n$ vertices labelled $\{1, \dots, n\}$. We can equip $Q$ with \emph{head} and \emph{tail} maps, which map an arrow $\alpha : i \to j$ to the vertex $j$ and the vertex $i$, respectively. A \emph{path (of length $\ell$}) in $Q$ is a sequence of arrows $p = \alpha_1 \dots \alpha_\ell$ such that $h(\alpha_i) = t(\alpha_{i+1})$ for $1 \leqslant i < \ell$ (in particular, we compose paths from left to right). We can extend the head and tail maps to paths in the obvious way. \\
\indent Given a finite quiver, we can form a $\Bbbk$-algebra $\Bbbk Q$ called the \emph{path algebra} of $Q$ as follows. As a vector space, $\Bbbk Q$ has a basis consisting of paths in the quiver (including the \emph{stationary paths} $e_i$ where we remain at vertex $i$), and multiplication of paths is given by concatenation, where defined: 
\begin{align*}
p \cdot q \coloneqq \left \{
\begin{array}{cl}
pq & \text{if } h(p) = t(q), \\
0 & \text{otherwise}.
\end{array}
\right.
\end{align*}
This is then extended linearly to all of $\Bbbk Q$. The elements $e_i$ are pairwise orthogonal idempotents, and the unit element in $\Bbbk Q$ is $1_{\Bbbk Q} = e_1 + e_2 + \dots + e_n$. The path algebra has a natural grading given by path length. \\
\indent If $\Bbbk Q$ is a path algebra, a \emph{relation} $\rho$ in $\Bbbk Q$ is an element of $(\Bbbk Q)_m$ for some $m \geqslant 2$, where every path in $\rho$ has the same head and tail, i.e.\ $\rho \in e_i (\Bbbk Q)_m e_j$ for some vertices $i$ and $j$. If $I$ is a two-sided ideal of $\Bbbk Q$ generated by relations, then we call $\Bbbk Q/I$ a \emph{path algebra with relations} or a \emph{quiver with relations}. \\
\indent By \cite[Theorem 4.8]{crawfordsuperpotentials}, it is possible to express $A \hash H$ as a quiver with relations. In particular, when $H = (\Bbbk G)^*$ and $A$ is AS regular of dimension $2$, we have the following result:

\begin{thm}[{\cite[Theorem 6.2]{crawfordsuperpotentials}}] \label{thm:QuiverAlgebra}
Suppose that $A = \Bbbk\langle u,v \rangle/\langle\rho(u,v) \rangle$ is a two-dimensional AS regular algebra and that $H = (\Bbbk G)^*$ acts on $A$ inner-faithfully and homogeneously. Moreover, assume that the action is $G$-homogeneous, and write
\begin{align*}
\deg_G u = a, \quad \deg_G v = b
\end{align*}
for some $a,b \in G$. Let $Q$ be the quiver whose vertices are labelled by the elements of $G$ and such that, for all $g \in G$, there are arrows
\begin{align*}
g \to a^{-1} g, \quad g \to b^{-1} g;
\end{align*}
additionally, decorate these arrows with $u$ and $v$, respectively. We call $Q$ the \emph{McKay quiver} for the action of $H$ on $A$. Finally, let $I = \langle \rho_g \mid g \in G \rangle$ be the ideal of $\Bbbk Q$ generated by the elements $\rho_g$ defined as follows: if we fix $g \in G$, then $\rho_g$ is the linear combination of paths of length $2$ from the vertex $g$ to the vertex $(\deg_G \rho(u,v))^{-1} g$ such that, if we replace an arrow in $\rho_g$ by its decoration, then we obtain $\rho(u,v)$. Then $A \hash H \cong \Lambda \coloneqq \Bbbk Q/I$ and $A^H \cong e_{1} \Lambda e_{1}$, where $e_{1}$ is the idempotent corresponding to the vertex labelled by the identity of $G$.
\end{thm}

\indent The McKay quiver mentioned above can be defined for the action of any Hopf algebra acting homogeneously on a graded algebra. %We remark that the decorations $u$ and $v$ in thew above the are simply a convenient tool which can be used to help identify the relations in $\Lambda$. 
Throughout this article, we will reserve the notation $\Lambda$ and $e_1$ for the algebra and vertex idempotent as in the statement of Theorem \ref{thm:QuiverAlgebra}. \\ %A careful analysis of the structure of the McKay quiver $Q$ is used to prove Theorem \ref{thm:intro3}. \\
\indent As mentioned previously, identifying when the Auslander map is an isomorphism is of particular interest. In general, it is difficult to determine when this is the case, but a result of Bao--He--Zhang \cite[Theorem 0.3]{bhz} provides a computational method which can be used in examples. The following is a restatement of their result that gives a first indication of how quivers will play a role in the proofs of results in this paper:

\begin{thm}[{\cite[Corollary 5.7]{crawfordsuperpotentials}}] \label{thm:FinDimLambda}
Assume the hypotheses of Theorem \ref{thm:QuiverAlgebra} and let $\Lambda$ be the algebra from this theorem. Then the Auslander map associated to the pair $(A,G)$ is an isomorphism if and only if $\Lambda/\langle e_{1} \rangle$ is finite dimensional.
\end{thm}

In many cases one can determine whether or not $\Lambda/\langle e_{1} \rangle$ is finite dimensional by using combinatorics on the quiver $Q$. We will use this method to prove Theorem \ref{thm:intro2}.

%%%%%%%%%%%%%%%%%%%%%%%%%%%%%%%%%%%%%%
%%%%%%%%%%%%%%%%%%%%%%%%%%%%%%%%%%%%%%
%%%%%%%%%%%%%%%%%%%%%%%%%%%%%%%%%%%%%%

\section{Classifying $G$-coactions}
We first classify the possible group coactions on two-dimensional AS regular algebras. If $G$ is abelian then $(\Bbbk G)^*$ is isomorphic to $\Bbbk G$ as Hopf algebras, and so the coaction is equivalent to the action by a finite abelian group; this case was considered in \cite{crawfordActions}. We therefore restrict attention to the case where $G$ is non-abelian. Additionally, \cite[Theorem 1.3]{etingof} shows that we may ignore the case where $A$ is a polynomial ring.

\begin{prop} \label{prop:Classifying}
Suppose that $A$ is a two-dimensional AS regular algebra which is not commutative and which is graded by a finite non-abelian group $G$ such that the corresponding action of $(\Bbbk G)^*$ is inner-faithful. Then $A \cong \Bbbk\langle u,v \rangle/\langle u^2-v^2 \rangle$ and $G = \langle a,b \mid \mathcal{R} \rangle$ for some set of relations $\mathcal{R}$, one of which is $a^2 = b^2$, and where the grading satisfies $\deg_G u = a$, $\deg_G v = b$.
\end{prop}
\begin{proof}
Recall that, up to isomorphism, the two-dimensional AS regular algebras are the quantum plane and the Jordan plane from (\ref{eqn:TwoDimASReg}); in the former case, if $q=-1$ then this algebra is isomorphic to the algebra appearing in the statement of the lemma. \\
\indent Let $A$ be either the quantum plane or the Jordan plane, and suppose that $A$ is group graded by a finite non-abelian group $G$. Since $G$ is not abelian and the coaction is inner-faithful, by Lemma \ref{lem:innerfaithful} it is generated by two elements $a$ and $b$, say. Without loss of generality, we may assume that the $G$-grading satisfies
\begin{align*}
\deg_G (u+\alpha v) = a, \quad \deg_G(\beta u + v) = b.
\end{align*}
for some $\alpha,\beta \in \Bbbk$ such that $\alpha \beta \neq 1$. Writing $x \coloneqq u+\alpha v$ and $y \coloneqq \beta u + v$, we have 
\begin{align*}
u = \frac{1}{1-\alpha \beta} (x-\alpha y), \qquad v = \frac{1}{1-\alpha\beta}(-\beta x + y).
\end{align*}
\indent Now assume that $A = \Bbbk_J[u,v]$. The defining relation for this algebra is then given by
\begin{align*}
vu-uv-u^2 = \frac{1}{(1-\alpha \beta)^2} \left( -x^2 + (\alpha\beta + \alpha-1)xy + (-\alpha \beta + \alpha + 1)yx - \alpha^2 y^2 \right).
\end{align*}
Note that the coefficients of $xy$ and $yx$ cannot both be $0$, else we deduce that $\Bbbk_J[x,y]$ is isomorphic to $\Bbbk\langle x,y \rangle/\langle x^2 + \alpha^2 y^2 \rangle$, which is isomorphic to $\Bbbk_{-1}[u,v]$. Therefore at least one of $xy$ or $yx$ must appear with nonzero coefficient; suppose it is the former, with the latter case being almost identical. For the action of $(\Bbbk G)^*$ on $A$ to be homogeneous, this relation must be $G$-homogeneous, and so it follows that
\begin{align*}
a^2 = \deg_G x^2 = \deg_G xy = ab.
\end{align*}
Therefore $a=b$, so that $G$ is in fact abelian, contrary to our assumption. \\
\indent Now suppose that $A = \Bbbk_q[u,v]$ for some $q \in \Bbbk^\times$. Since we have assumed that $A$ is not commutative, we have $q \neq 1$. In this case, the defining relation becomes
\begin{align*}
vu-quv = \frac{1}{(1-\alpha \beta)^2} \left( (q \beta-\beta)x^2 + (\alpha \beta-q)xy + (1-q\alpha \beta)yx + (q\alpha-\alpha)y^2 \right).
\end{align*}
Since $q \neq 1$, the coefficients of $x^2$ and $y^2$ are both nonzero. As with the Jordan plane, if the coefficient of $xy$ or $yx$ is nonzero, then $G$ is abelian. Therefore we must have $\alpha \beta - q = 0 = 1-q \alpha \beta$, so that $q^2=1$ and hence $q=-1=\alpha \beta$. Finally, if we define 
\begin{align*}
X \coloneqq \sqrt{\beta} x, \qquad Y \coloneqq \sqrt{-\alpha} \gap y, 
\end{align*}
then the defining relation becomes
\begin{align*}
vu-quv = -\frac{1}{2}(-2 \beta x^2 - 2\alpha y^2) = X^2-Y^2,
\end{align*}
where $\deg_G X = a$ and $\deg_G Y = b$. Since the relation is $G$-homogeneous, we must also have
\begin{align*}
a^2 = \deg_G X^2 = \deg_G Y^2 = b^2, 
\end{align*}
and the result now follows.
\end{proof}

%%%%%%%%%%%%%%%%%%%%%%%%%%%%%%%%%%%%%%
%%%%%%%%%%%%%%%%%%%%%%%%%%%%%%%%%%%%%%
%%%%%%%%%%%%%%%%%%%%%%%%%%%%%%%%%%%%%%

\section{The McKay Quiver and Combinatorics} \label{sec:Quivers}
Before proving our main results, we provide an extended example to help demonstrate some of the combinatorics concerning the algebras $\Lambda$ from Theorem \ref{thm:QuiverAlgebra}. \\
\indent With Proposition \ref{prop:Classifying} in hand, we are interested in group coactions on the algebra
\begin{align*}
A = \frac{\Bbbk\langle u,v \rangle }{\langle u^2-v^2 \rangle}.
\end{align*}
In particular, $G$ must satisfy the hypotheses of Proposition \ref{prop:Classifying}, so in this section let
\begin{align*}
G = \langle a,b \mid a^2=b^2, a^4=b^4=(ab)^3=1 \rangle,
\end{align*}
and equip $A$ with a $G$-grading by declaring
\begin{align*}
\deg_G u = a \quad \text{and} \quad \deg_G v = b.
\end{align*}
The group $G$ is isomorphic to the binary dihedral group of order $12$ and, as a set, we have
\begin{align*}
G = \{ 1,a,a^2,a^3, b,ab,a^2b,a^3b, ba,aba,a^2ba,a^3ba \}. %\label{GElts} 
\end{align*}
Seeking to determine properties of $A^H$ and $A \hash H$, we construct the algebra $\Lambda$ from Theorem \ref{thm:QuiverAlgebra}. Following the recipe from this theorem, a straightforward computation verifies that the McKay quiver $Q$ is as shown below on the left, where the vertices are labelled by the elements of $G$ and the arrows are decorated with the elements $u$ and $v$. We also provide a relabelling on the right, to allow us to more easily talk about the elements of $\Lambda$:

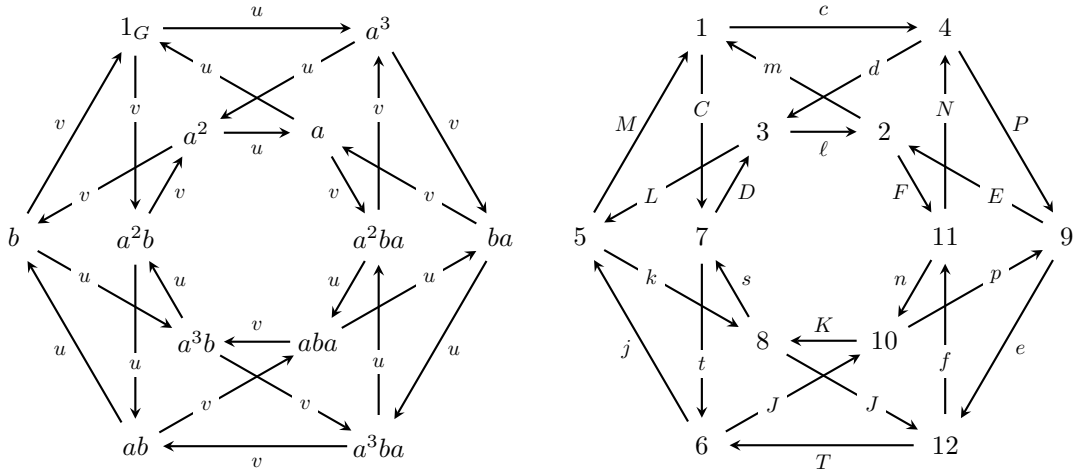
\begin{equation*}
\begin{tikzpicture}[->,>=stealth,thick,scale=1]
\def \wiggle {1.6};

\node [minimum size=0.7cm,circle](1) at (0:\wiggle) {};
\node [minimum size=0.7cm,circle](2) at (60:\wiggle) {};
\node [minimum size=0.7cm,circle](3) at (120:\wiggle) {};
\node [minimum size=0.7cm,circle](4) at (180:\wiggle) {};
\node [minimum size=0.7cm,circle](5) at (240:\wiggle) {};
\node [minimum size=0.7cm,circle](6) at (300:\wiggle) {};

\node [minimum size=0.7cm,circle](7) at (0:2*\wiggle) {};
\node [minimum size=0.7cm,circle](8) at (60:2*\wiggle) {};
\node [minimum size=0.7cm,circle](9) at (120:2*\wiggle) {};
\node [minimum size=0.7cm,circle](10) at (180:2*\wiggle) {};
\node [minimum size=0.7cm,circle](11) at (240:2*\wiggle) {};
\node [minimum size=0.8cm,circle](12) at (300:2*\wiggle) {};

\node at (0-120:\wiggle) {$a^3b$};
\node at (60-120:\wiggle) {$aba$};
\node at (120-120:\wiggle) {$a^2ba$};
\node at (180-120:\wiggle) {$a$};
\node at (240-120:\wiggle) {$a^2$};
\node at (300-120:\wiggle) {$a^2b$};

\node at (0-120:2*\wiggle) {$ab$};
\node at (60-120:2*\wiggle) {$a^3ba$};
\node at (120-120:2*\wiggle) {$ba$};
\node at (180-120:2*\wiggle) {$a^3$};
\node at (240-120:2*\wiggle) {$1_G$};
\node at (300-120:2*\wiggle) {$b$};

%\draw (6) to node[pos=0.5,scale=0.8,right]{$v$} (3);

\draw (1) to (6);
\draw (6) to (5);
\draw (5) to (4);
\draw (4) to (3);
\draw (3) to (2);
\draw (2) to (1);

\node[scale=0.8] at (0-30:0.73*\wiggle) {$u$};
\node[scale=0.8] at (60-30:0.73*\wiggle) {$v$};
\node[scale=0.8] at (120-30:0.73*\wiggle) {$u$};
\node[scale=0.8] at (180-30:0.73*\wiggle) {$v$};
\node[scale=0.8] at (240-30:0.73*\wiggle) {$u$};
\node[scale=0.8] at (300-30:0.73*\wiggle) {$v$};

\draw (7) to  (12);
\draw (12) to (11);
\draw (11) to (10);
\draw (10) to (9);
\draw (9) to (8);
\draw (8) to (7);

\node[scale=0.8] at (0-30:1.87*\wiggle) {$u$};
\node[scale=0.8] at (60-30:1.87*\wiggle) {$v$};
\node[scale=0.8] at (120-30:1.87*\wiggle) {$u$};
\node[scale=0.8] at (180-30:1.87*\wiggle) {$v$};
\node[scale=0.8] at (240-30:1.87*\wiggle) {$u$};
\node[scale=0.8] at (300-30:1.87*\wiggle) {$v$};

\draw (11) to node[pos=0.35,scale=0.8,fill=white]{$v$} (6);
\draw (6) to node[pos=0.65,scale=0.8,fill=white]{$u$} (7);
\draw (7) to node[pos=0.35,scale=0.8,fill=white]{$v$} (2);
\draw (2) to node[pos=0.65,scale=0.8,fill=white]{$u$} (9);
\draw (9) to node[pos=0.35,scale=0.8,fill=white]{$v$} (4);
\draw (4) to node[pos=0.65,scale=0.8,fill=white]{$u$} (11);

\draw (12) to node[pos=0.35,scale=0.8,fill=white]{$u$} (1);
\draw (1) to node[pos=0.65,scale=0.8,fill=white]{$v$} (8);
\draw (8) to node[pos=0.35,scale=0.8,fill=white]{$u$} (3);
\draw (3) to node[pos=0.65,scale=0.8,fill=white]{$v$} (10);
\draw (10) to node[pos=0.35,scale=0.8,fill=white]{$u$} (5);
\draw (5) to node[pos=0.65,scale=0.8,fill=white]{$v$} (12);

\end{tikzpicture} 
\hspace{10pt}
\begin{tikzpicture}[->,>=stealth,thick,scale=1]
\def \wiggle {1.6};

\node [minimum size=0.7cm,circle](1) at (0:\wiggle) {};
\node [minimum size=0.7cm,circle](2) at (60:\wiggle) {};
\node [minimum size=0.7cm,circle](3) at (120:\wiggle) {};
\node [minimum size=0.7cm,circle](4) at (180:\wiggle) {};
\node [minimum size=0.7cm,circle](5) at (240:\wiggle) {};
\node [minimum size=0.7cm,circle](6) at (300:\wiggle) {};

\node [minimum size=0.7cm,circle](7) at (0:2*\wiggle) {};
\node [minimum size=0.7cm,circle](8) at (60:2*\wiggle) {};
\node [minimum size=0.7cm,circle](9) at (120:2*\wiggle) {};
\node [minimum size=0.7cm,circle](10) at (180:2*\wiggle) {};
\node [minimum size=0.7cm,circle](11) at (240:2*\wiggle) {};
\node [minimum size=0.8cm,circle](12) at (300:2*\wiggle) {};

% \node at (0:\wiggle) {$a^3b$};
% \node at (60:\wiggle) {$aba$};
% \node at (120:\wiggle) {$a^2ba$};
% \node at (180:\wiggle) {$a^2$};
% \node at (240:\wiggle) {$a^3$};
% \node at (300:\wiggle) {$a^2b$};

% \node at (0:2*\wiggle) {$ab$};
% \node at (60:2*\wiggle) {$a^3ba$};
% \node at (120:2*\wiggle) {$ba$};
% \node at (180:2*\wiggle) {$a^3$};
% \node at (240:2*\wiggle) {$1_G$};
% \node at (300:2*\wiggle) {$b$};

\node at (0-120:\wiggle) {$8$};
\node at (60-120:\wiggle) {$10$};
\node at (120-120:\wiggle) {$11$};
\node at (180-120:\wiggle) {$2$};
\node at (240-120:\wiggle) {$3$};
\node at (300-120:\wiggle) {$7$};

\node at (0-120:2*\wiggle) {$6$};
\node at (60-120:2*\wiggle) {$12$};
\node at (120-120:2*\wiggle) {$9$};
\node at (180-120:2*\wiggle) {$4$};
\node at (240-120:2*\wiggle) {$1$};
\node at (300-120:2*\wiggle) {$5$};

%\draw (6) to node[pos=0.5,scale=0.8,right]{$v$} (3);

\draw (1) to (6);
\draw (6) to (5);
\draw (5) to (4);
\draw (4) to (3);
\draw (3) to (2);
\draw (2) to (1);

\node[scale=0.8] at (0-30:0.73*\wiggle) {$n$};
\node[scale=0.8] at (60-30:0.73*\wiggle) {$F$};
\node[scale=0.8] at (120-30:0.73*\wiggle) {$\ell$};
\node[scale=0.8] at (180-30:0.73*\wiggle) {$D$};
\node[scale=0.8] at (240-30:0.73*\wiggle) {$s$};
\node[scale=0.8] at (300-30:0.73*\wiggle) {$K$};

\draw (7) to  (12);
\draw (12) to (11);
\draw (11) to (10);
\draw (10) to (9);
\draw (9) to (8);
\draw (8) to (7);

\node[scale=0.8] at (0-30:1.87*\wiggle) {$e$};
\node[scale=0.8] at (60-30:1.87*\wiggle) {$P$};
\node[scale=0.8] at (120-30:1.87*\wiggle) {$c$};
\node[scale=0.8] at (180-30:1.87*\wiggle) {$M$};
\node[scale=0.8] at (240-30:1.87*\wiggle) {$j$};
\node[scale=0.8] at (300-30:1.87*\wiggle) {$T$};

\draw (11) to node[pos=0.35,scale=0.8,fill=white]{$J$} (6);
\draw (6) to node[pos=0.65,scale=0.8,fill=white]{$p$} (7);
\draw (7) to node[pos=0.35,scale=0.8,fill=white]{$E$} (2);
\draw (2) to node[pos=0.65,scale=0.8,fill=white]{$m$} (9);
\draw (9) to node[pos=0.35,scale=0.8,fill=white]{$C$} (4);
\draw (4) to node[pos=0.65,scale=0.8,fill=white]{$t$} (11);

\draw (12) to node[pos=0.35,scale=0.8,fill=white]{$f$} (1);
\draw (1) to node[pos=0.65,scale=0.8,fill=white]{$N$} (8);
\draw (8) to node[pos=0.35,scale=0.8,fill=white]{$d$} (3);
\draw (3) to node[pos=0.65,scale=0.8,fill=white]{$L$} (10);
\draw (10) to node[pos=0.35,scale=0.8,fill=white]{$k$} (5);
\draw (5) to node[pos=0.65,scale=0.8,fill=white]{$J$} (12);

\end{tikzpicture}\label{McKayExample}
\end{equation*}

\indent To define the algebra $\Lambda$ to which $A \hash (\Bbbk G)^*$ is isomorphic, it remains to determine the relations in the quiver. Since the defining relation of $A$ is $u^2-v^2$, which satisfies $\deg_G (u^2-v^2) = a^2$, the relations in $\Lambda$ are precisely the paths in the quiver which start at a vertex $g$ and end at $a^{-2} g = a^2 g$ and which trace out the relation $u^2=v^2$. For example, by comparing the two different labellings, we see that $cd = CD$ and $d\ell = PE$ are both relations in $\Lambda$; there are twelve relations in total, one for each starting vertex (equivalently, each group element). \\
\indent While the above quiver is aesthetically pleasing, it is difficult to determine when two paths of length $\ell \geqslant 2$ are equivalent in $\Lambda$. It turns out that it is more convenient to visualise the McKay quiver as a suitable quotient of the quiver $\mathbb{Z} \mathbb{A}_\infty^\infty$. This is the infinite quiver whose vertices form a square lattice in the plane, and our convention will be that the arrows point south and east towards neighbouring vertices. In the running example, the McKay quiver may be viewed as follows, where we omit the arrow labels, but colour the arrows red or blue when the corresponding label is $u$ or $v$, respectively:

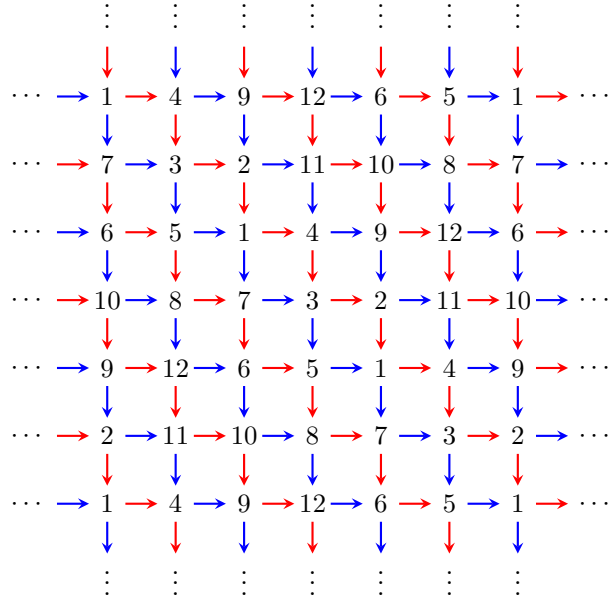
\begin{equation*}
\begin{tikzpicture}[->,>=stealth,thick,scale=1]
\def \wiggle {0.9};

\node at (0*\wiggle,0*\wiggle) {$1$};
\node at (1*\wiggle,0*\wiggle) {$4$};
\node at (2*\wiggle,0*\wiggle) {$9$};
\node at (3*\wiggle,0*\wiggle) {$12$};
\node at (4*\wiggle,0*\wiggle) {$6$};
\node at (5*\wiggle,0*\wiggle) {$5$};
\node at (6*\wiggle,0*\wiggle) {$1$};

\node at (0*\wiggle,-1*\wiggle) {$7$};
\node at (1*\wiggle,-1*\wiggle) {$3$};
\node at (2*\wiggle,-1*\wiggle) {$2$};
\node at (3*\wiggle,-1*\wiggle) {$11$};
\node at (4*\wiggle,-1*\wiggle) {$10$};
\node at (5*\wiggle,-1*\wiggle) {$8$};
\node at (6*\wiggle,-1*\wiggle) {$7$};

\node at (0*\wiggle,-2*\wiggle) {$6$};
\node at (1*\wiggle,-2*\wiggle) {$5$};
\node at (2*\wiggle,-2*\wiggle) {$1$};
\node at (3*\wiggle,-2*\wiggle) {$4$};
\node at (4*\wiggle,-2*\wiggle) {$9$};
\node at (5*\wiggle,-2*\wiggle) {$12$};
\node at (6*\wiggle,-2*\wiggle) {$6$};

\node at (0*\wiggle,-3*\wiggle) {$10$};
\node at (1*\wiggle,-3*\wiggle) {$8$};
\node at (2*\wiggle,-3*\wiggle) {$7$};
\node at (3*\wiggle,-3*\wiggle) {$3$};
\node at (4*\wiggle,-3*\wiggle) {$2$};
\node at (5*\wiggle,-3*\wiggle) {$11$};
\node at (6*\wiggle,-3*\wiggle) {$10$};

\node at (0*\wiggle,-4*\wiggle) {$9$};
\node at (1*\wiggle,-4*\wiggle) {$12$};
\node at (2*\wiggle,-4*\wiggle) {$6$};
\node at (3*\wiggle,-4*\wiggle) {$5$};
\node at (4*\wiggle,-4*\wiggle) {$1$};
\node at (5*\wiggle,-4*\wiggle) {$4$};
\node at (6*\wiggle,-4*\wiggle) {$9$};

\node at (0*\wiggle,-5*\wiggle) {$2$};
\node at (1*\wiggle,-5*\wiggle) {$11$};
\node at (2*\wiggle,-5*\wiggle) {$10$};
\node at (3*\wiggle,-5*\wiggle) {$8$};
\node at (4*\wiggle,-5*\wiggle) {$7$};
\node at (5*\wiggle,-5*\wiggle) {$3$};
\node at (6*\wiggle,-5*\wiggle) {$2$};

\node at (0*\wiggle,-6*\wiggle) {$1$};
\node at (1*\wiggle,-6*\wiggle) {$4$};
\node at (2*\wiggle,-6*\wiggle) {$9$};
\node at (3*\wiggle,-6*\wiggle) {$12$};
\node at (4*\wiggle,-6*\wiggle) {$6$};
\node at (5*\wiggle,-6*\wiggle) {$5$};
\node at (6*\wiggle,-6*\wiggle) {$1$};

%\node at (-1*\wiggle,1*\wiggle) {$100$};
\node at (0*\wiggle,1.3*\wiggle) {$\vdots$};
\node at (1*\wiggle,1.3*\wiggle) {$\vdots$};
\node at (2*\wiggle,1.3*\wiggle) {$\vdots$};
\node at (3*\wiggle,1.3*\wiggle) {$\vdots$};
\node at (4*\wiggle,1.3*\wiggle) {$\vdots$};
\node at (5*\wiggle,1.3*\wiggle) {$\vdots$};
\node at (6*\wiggle,1.3*\wiggle) {$\vdots$};
%\node (108) at (7*\wiggle,1*\wiggle) {$108$};

%\node (110) at (-1*\wiggle,-7*\wiggle) {$110$};
\node at (0*\wiggle,-7.05*\wiggle) {$\vdots$};
\node at (1*\wiggle,-7.05*\wiggle) {$\vdots$};
\node at (2*\wiggle,-7.05*\wiggle) {$\vdots$};
\node at (3*\wiggle,-7.05*\wiggle) {$\vdots$};
\node at (4*\wiggle,-7.05*\wiggle) {$\vdots$};
\node at (5*\wiggle,-7.05*\wiggle) {$\vdots$};
\node at (6*\wiggle,-7.05*\wiggle) {$\vdots$};
%\node (118) at (7*\wiggle,-7*\wiggle) {$118$};

\node at (-1.15*\wiggle,0*\wiggle) {$\cdots$};
\node at (-1.15*\wiggle,-1*\wiggle) {$\cdots$};
\node at (-1.15*\wiggle,-2*\wiggle) {$\cdots$};
\node at (-1.15*\wiggle,-3*\wiggle) {$\cdots$};
\node at (-1.15*\wiggle,-4*\wiggle) {$\cdots$};
\node at (-1.15*\wiggle,-5*\wiggle) {$\cdots$};
\node at (-1.15*\wiggle,-6*\wiggle) {$\cdots$};

\node at (7.15*\wiggle,0*\wiggle) {$\cdots$};
\node at (7.15*\wiggle,-1*\wiggle) {$\cdots$};
\node at (7.15*\wiggle,-2*\wiggle) {$\cdots$};
\node at (7.15*\wiggle,-3*\wiggle) {$\cdots$};
\node at (7.15*\wiggle,-4*\wiggle) {$\cdots$};
\node at (7.15*\wiggle,-5*\wiggle) {$\cdots$};
\node at (7.15*\wiggle,-6*\wiggle) {$\cdots$};

\node [minimum width=0.45cm,minimum height=0.45cm](1) at (0*\wiggle,0*\wiggle) {};
\node [minimum width=0.45cm,minimum height=0.45cm](2) at (1*\wiggle,0*\wiggle) {};
\node [minimum width=0.45cm,minimum height=0.45cm](3) at (2*\wiggle,0*\wiggle) {};
\node [minimum width=0.45cm,minimum height=0.45cm](4) at (3*\wiggle,0*\wiggle) {};
\node [minimum width=0.45cm,minimum height=0.45cm](5) at (4*\wiggle,0*\wiggle) {};
\node [minimum width=0.45cm,minimum height=0.45cm](6) at (5*\wiggle,0*\wiggle) {};
\node [minimum width=0.45cm,minimum height=0.45cm](7) at (6*\wiggle,0*\wiggle) {};

\node [minimum width=0.45cm,minimum height=0.45cm](8) at (0*\wiggle,-1*\wiggle) {};
\node [minimum width=0.45cm,minimum height=0.45cm](9) at (1*\wiggle,-1*\wiggle) {};
\node [minimum width=0.45cm,minimum height=0.45cm](10) at (2*\wiggle,-1*\wiggle) {};
\node [minimum width=0.45cm,minimum height=0.45cm](11) at (3*\wiggle,-1*\wiggle) {};
\node [minimum width=0.45cm,minimum height=0.45cm](12) at (4*\wiggle,-1*\wiggle) {};
\node [minimum width=0.45cm,minimum height=0.45cm](13) at (5*\wiggle,-1*\wiggle) {};
\node [minimum width=0.45cm,minimum height=0.45cm](14) at (6*\wiggle,-1*\wiggle) {};

\node [minimum width=0.45cm,minimum height=0.45cm](15) at (0*\wiggle,-2*\wiggle) {};
\node [minimum width=0.45cm,minimum height=0.45cm](16) at (1*\wiggle,-2*\wiggle) {};
\node [minimum width=0.45cm,minimum height=0.45cm](17) at (2*\wiggle,-2*\wiggle) {};
\node [minimum width=0.45cm,minimum height=0.45cm](18) at (3*\wiggle,-2*\wiggle) {};
\node [minimum width=0.45cm,minimum height=0.45cm](19) at (4*\wiggle,-2*\wiggle) {};
\node [minimum width=0.45cm,minimum height=0.45cm](20) at (5*\wiggle,-2*\wiggle) {};
\node [minimum width=0.45cm,minimum height=0.45cm](21) at (6*\wiggle,-2*\wiggle) {};

\node [minimum width=0.45cm,minimum height=0.45cm](22) at (0*\wiggle,-3*\wiggle) {};
\node [minimum width=0.45cm,minimum height=0.45cm](23) at (1*\wiggle,-3*\wiggle) {};
\node [minimum width=0.45cm,minimum height=0.45cm](24) at (2*\wiggle,-3*\wiggle) {};
\node [minimum width=0.45cm,minimum height=0.45cm](25) at (3*\wiggle,-3*\wiggle) {};
\node [minimum width=0.45cm,minimum height=0.45cm](26) at (4*\wiggle,-3*\wiggle) {};
\node [minimum width=0.45cm,minimum height=0.45cm](27) at (5*\wiggle,-3*\wiggle) {};
\node [minimum width=0.45cm,minimum height=0.45cm](28) at (6*\wiggle,-3*\wiggle) {};

\node [minimum width=0.45cm,minimum height=0.45cm](29) at (0*\wiggle,-4*\wiggle) {};
\node [minimum width=0.45cm,minimum height=0.45cm](30) at (1*\wiggle,-4*\wiggle) {};
\node [minimum width=0.45cm,minimum height=0.45cm](31) at (2*\wiggle,-4*\wiggle) {};
\node [minimum width=0.45cm,minimum height=0.45cm](32) at (3*\wiggle,-4*\wiggle) {};
\node [minimum width=0.45cm,minimum height=0.45cm](33) at (4*\wiggle,-4*\wiggle) {};
\node [minimum width=0.45cm,minimum height=0.45cm](34) at (5*\wiggle,-4*\wiggle) {};
\node [minimum width=0.45cm,minimum height=0.45cm](35) at (6*\wiggle,-4*\wiggle) {};

\node [minimum width=0.45cm,minimum height=0.45cm](36) at (0*\wiggle,-5*\wiggle) {};
\node [minimum width=0.45cm,minimum height=0.45cm](37) at (1*\wiggle,-5*\wiggle) {};
\node [minimum width=0.45cm,minimum height=0.45cm](38) at (2*\wiggle,-5*\wiggle) {};
\node [minimum width=0.45cm,minimum height=0.45cm](39) at (3*\wiggle,-5*\wiggle) {};
\node [minimum width=0.45cm,minimum height=0.45cm](40) at (4*\wiggle,-5*\wiggle) {};
\node [minimum width=0.45cm,minimum height=0.45cm](41) at (5*\wiggle,-5*\wiggle) {};
\node [minimum width=0.45cm,minimum height=0.45cm](42) at (6*\wiggle,-5*\wiggle) {};

\node [minimum width=0.45cm,minimum height=0.45cm](43) at (0*\wiggle,-6*\wiggle) {};
\node [minimum width=0.45cm,minimum height=0.45cm](44) at (1*\wiggle,-6*\wiggle) {};
\node [minimum width=0.45cm,minimum height=0.45cm](45) at (2*\wiggle,-6*\wiggle) {};
\node [minimum width=0.45cm,minimum height=0.45cm](46) at (3*\wiggle,-6*\wiggle) {};
\node [minimum width=0.45cm,minimum height=0.45cm](47) at (4*\wiggle,-6*\wiggle) {};
\node [minimum width=0.45cm,minimum height=0.45cm](48) at (5*\wiggle,-6*\wiggle) {};
\node [minimum width=0.45cm,minimum height=0.45cm](49) at (6*\wiggle,-6*\wiggle) {};

%\node at (-1*\wiggle,1*\wiggle) {$100$};
\node [minimum width=0.45cm,minimum height=0.45cm](101) at (0*\wiggle,1*\wiggle) {};
\node [minimum width=0.45cm,minimum height=0.45cm](102) at (1*\wiggle,1*\wiggle) {};
\node [minimum width=0.45cm,minimum height=0.45cm](103) at (2*\wiggle,1*\wiggle) {};
\node [minimum width=0.45cm,minimum height=0.45cm](104) at (3*\wiggle,1*\wiggle) {};
\node [minimum width=0.45cm,minimum height=0.45cm](105) at (4*\wiggle,1*\wiggle) {};
\node [minimum width=0.45cm,minimum height=0.45cm](106) at (5*\wiggle,1*\wiggle) {};
\node [minimum width=0.45cm,minimum height=0.45cm](107) at (6*\wiggle,1*\wiggle) {};
%\node (108) at (7*\wiggle,1*\wiggle) {$108$};

%\node (110) at (-1*\wiggle,-7*\wiggle) {$110$};
\node [minimum width=0.45cm,minimum height=0.45cm](111) at (0*\wiggle,-7*\wiggle) {};
\node [minimum width=0.45cm,minimum height=0.45cm](112) at (1*\wiggle,-7*\wiggle) {};
\node [minimum width=0.45cm,minimum height=0.45cm](113) at (2*\wiggle,-7*\wiggle) {};
\node [minimum width=0.45cm,minimum height=0.45cm](114) at (3*\wiggle,-7*\wiggle) {};
\node [minimum width=0.45cm,minimum height=0.45cm](115) at (4*\wiggle,-7*\wiggle) {};
\node [minimum width=0.45cm,minimum height=0.45cm](116) at (5*\wiggle,-7*\wiggle) {};
\node [minimum width=0.45cm,minimum height=0.45cm](117) at (6*\wiggle,-7*\wiggle) {};
%\node (118) at (7*\wiggle,-7*\wiggle) {$118$};

\node [minimum width=0.45cm,minimum height=0.45cm](121) at (-1*\wiggle,0*\wiggle) {};
\node [minimum width=0.45cm,minimum height=0.45cm](122) at (-1*\wiggle,-1*\wiggle) {};
\node [minimum width=0.45cm,minimum height=0.45cm](123) at (-1*\wiggle,-2*\wiggle) {};
\node [minimum width=0.45cm,minimum height=0.45cm](124) at (-1*\wiggle,-3*\wiggle) {};
\node [minimum width=0.45cm,minimum height=0.45cm](125) at (-1*\wiggle,-4*\wiggle) {};
\node [minimum width=0.45cm,minimum height=0.45cm](126) at (-1*\wiggle,-5*\wiggle) {};
\node [minimum width=0.45cm,minimum height=0.45cm](127) at (-1*\wiggle,-6*\wiggle) {};

\node [minimum width=0.45cm,minimum height=0.45cm](131) at (7*\wiggle,0*\wiggle) {};
\node [minimum width=0.45cm,minimum height=0.45cm](132) at (7*\wiggle,-1*\wiggle) {};
\node [minimum width=0.45cm,minimum height=0.45cm](133) at (7*\wiggle,-2*\wiggle) {};
\node [minimum width=0.45cm,minimum height=0.45cm](134) at (7*\wiggle,-3*\wiggle) {};
\node [minimum width=0.45cm,minimum height=0.45cm](135) at (7*\wiggle,-4*\wiggle) {};
\node [minimum width=0.45cm,minimum height=0.45cm](136) at (7*\wiggle,-5*\wiggle) {};
\node [minimum width=0.45cm,minimum height=0.45cm](137) at (7*\wiggle,-6*\wiggle) {};

\draw[red] (101) -- (1);
\draw[blue] (102) -- (2);
\draw[red] (103) -- (3);
\draw[blue] (104) -- (4);
\draw[red] (105) -- (5);
\draw[blue] (106) -- (6);
\draw[red] (107) -- (7);

\draw[blue] (43) -- (111);
\draw[red] (44) -- (112);
\draw[blue] (45) -- (113);
\draw[red] (46) -- (114);
\draw[blue] (47) -- (115);
\draw[red] (48) -- (116);
\draw[blue] (49) -- (117);

\draw[blue] (121) -- (1);
\draw[red] (122) -- (8);
\draw[blue] (123) -- (15);
\draw[red] (124) -- (22);
\draw[blue] (125) -- (29);
\draw[red] (126) -- (36);
\draw[blue] (127) -- (43);

\draw[red] (7) -- (131);
\draw[blue] (14) -- (132);
\draw[red] (21) -- (133);
\draw[blue] (28) -- (134);
\draw[red] (35) -- (135);
\draw[blue] (42) -- (136);
\draw[red] (49) -- (137);

\draw[red] (1) -- (2);
\draw[blue] (2) -- (3);
\draw[red] (3) -- (4);
\draw[blue] (4) -- (5);
\draw[red] (5) -- (6);
\draw[blue] (6) -- (7);

\draw[blue] (8) -- (9);
\draw[red] (9) -- (10);
\draw[blue] (10) -- (11);
\draw[red] (11) -- (12);
\draw[blue] (12) -- (13);
\draw[red] (13) -- (14);

\draw[red] (15) -- (16);
\draw[blue] (16) -- (17);
\draw[red] (17) -- (18);
\draw[blue] (18) -- (19);
\draw[red] (19) -- (20);
\draw[blue] (20) -- (21);

\draw[blue] (22) -- (23);
\draw[red] (23) -- (24);
\draw[blue] (24) -- (25);
\draw[red] (25) -- (26);
\draw[blue] (26) -- (27);
\draw[red] (27) -- (28);

\draw[red] (29) -- (30);
\draw[blue] (30) -- (31);
\draw[red] (31) -- (32);
\draw[blue] (32) -- (33);
\draw[red] (33) -- (34);
\draw[blue] (34) -- (35);

\draw[blue] (36) -- (37);
\draw[red] (37) -- (38);
\draw[blue] (38) -- (39);
\draw[red] (39) -- (40);
\draw[blue] (40) -- (41);
\draw[red] (41) -- (42);

\draw[red] (43) -- (44);
\draw[blue] (44) -- (45);
\draw[red] (45) -- (46);
\draw[blue] (46) -- (47);
\draw[red] (47) -- (48);
\draw[blue] (48) -- (49);

\draw[blue] (1) -- (8);
\draw[red] (2) -- (9);
\draw[blue] (3) -- (10);
\draw[red] (4) -- (11);
\draw[blue] (5) -- (12);
\draw[red] (6) -- (13);
\draw[blue] (7) -- (14);

\draw[red] (8) -- (15);
\draw[blue] (9) -- (16);
\draw[red] (10) -- (17);
\draw[blue] (11) -- (18);
\draw[red] (12) -- (19);
\draw[blue] (13) -- (20);
\draw[red] (14) -- (21);

\draw[blue] (15) -- (22);
\draw[red] (16) -- (23);
\draw[blue] (17) -- (24);
\draw[red] (18) -- (25);
\draw[blue] (19) -- (26);
\draw[red] (20) -- (27);
\draw[blue] (21) -- (28);

\draw[red] (22) -- (29);
\draw[blue] (23) -- (30);
\draw[red] (24) -- (31);
\draw[blue] (25) -- (32);
\draw[red] (26) -- (33);
\draw[blue] (27) -- (34);
\draw[red] (28) -- (35);

\draw[blue] (29) -- (36);
\draw[red] (30) -- (37);
\draw[blue] (31) -- (38);
\draw[red] (32) -- (39);
\draw[blue] (33) -- (40);
\draw[red] (34) -- (41);
\draw[blue] (35) -- (42);

\draw[red] (36) -- (43);
\draw[blue] (37) -- (44);
\draw[red] (38) -- (45);
\draw[blue] (39) -- (46);
\draw[red] (40) -- (47);
\draw[blue] (41) -- (48);
\draw[red] (42) -- (49);

\end{tikzpicture}
\end{equation*}

\noindent Here, vertices sharing a number are identified, so we need only show a small portion of the quiver $\mathbb{Z} \mathbb{A}_\infty^\infty$ which then repeats infinitely. We call this the \emph{lattice presentation} of the McKay quiver. \\
\indent In the portion of infinite quiver shown above, observe that opposite sides can be identified, so it suffices to present just the square grid and omit the flanking arrows. Provided that $A$ and $G$ satisfy the hypotheses of Proposition \ref{prop:Classifying}, it is always possible to present the McKay quiver in this way. If we additionally assume that the top left vertex is labelled by $1_G$, then We call this the \emph{toroidal presentation} of the McKay quiver. \\
\indent Observe that, in each row and column, the arrows alternate between red and blue, corresponding to the fact that the decorations alternate between $u$ and $v$. In our presentation of the toroidal McKay quiver, our convention will be that the first arrow in the top row is decorated by $u$, so that the first arrow in the left hand column is decorated by $v$. Equivalently, this says that the vertices in the top row of the toroidal McKay quiver are labelled by (from left to right) $1_G$, $a^{-1}$, $b^{-1} a^{-1}$, $a^{-1} b^{-1} a^{-1}, \dots, 1_G$. Working with the lattice presentation of the McKay quiver, the relations in $\Lambda$ say that the two paths of length two from a vertex to the vertex immediately south-east are equal; by repeatedly applying these relations, it follows that there is precisely one nonzero path in $\Lambda$ from a vertex to any vertex which is south and/or east of it. \\
\indent As an example, the paths $cPeTjM$ and $CtJpEm$ in $\Lambda$ are distinct because, despite both having common start and terminal vertices and having the same length, they visibly correspond to different paths in the lattice presentation of the McKay quiver. On the other hand, the paths $cPEm$ and $CtjM$ are equal in $\Lambda$, which again is clear if we use this presentation. Neither of these facts are obvious using our original presentation of $\Lambda$. \\
\indent The lattice and toroidal presentations make it particularly easy to perform computations in $\Lambda$. For example, using this description, it is straightforward to determine the invariant ring $A^H = A_{1_G}$. By Theorem \ref{thm:QuiverAlgebra}, we have $A^H \cong e_{1} \Lambda e_{1}$, so it suffices to look at paths from the vertex $1$ to itself in the quiver. As noted above, there are distinct paths $aPeTjM$ and $CtJpEm$, as well as a path of length four given by $cd\ell m$ (and five other equivalent paths). From the toroidal presentation of $\Lambda$, one can show that any other path from the vertex $1$ to itself can be written as a product of these three paths. Calling these paths $x$, $y$, and $z$, respectively, it is easy to check that these elements commute and satisfy $xy = z^3$; with a little more work, we obtain
\begin{align*}
A^H \cong \frac{\Bbbk[x,y,z]}{\langle xy-z^3 \rangle},
\end{align*}
which is (the coordinate ring of) an $\mathbb{A}_2$ singularity. Working with $A$ rather than $\Lambda$, the elements $x$, $y$, and $z$ correspond to $(uv)^3, (vu)^3$, and $u^4$, respectively, which also satisfy the relation $xy=z^3$.

\subsection{Dimension counting via quiver combinatorics} \label{sec:DimensionCounting}
We now describe a method to count the dimensions of various subspaces of the algebra $\Lambda$ associated to the group coaction of $G$ on $A$. This is most easily described by continuing with our previous example. The methods described in this section are specific examples of a technique described in \cite[Section 4]{gl2c}. \\
\indent Suppose we are interested in determining the number of paths in $\Lambda$ from a vertex $i$ to a vertex $j$ of a certain length $\ell$; equivalently, we are interested in $\dim_\Bbbk (e_i \Lambda e_j)_\ell$. Now, as observed earlier, the relations in $\Lambda$ tell us that there is exactly one path from a vertex to any vertex south and/or east of it in the lattice presentation of the McKay quiver. Seeking to compute $\dim_\Bbbk (e_2 \Lambda e_j)_\ell$, for example, consider the following diagram, where one should ignore the shading and the red numbers for now:

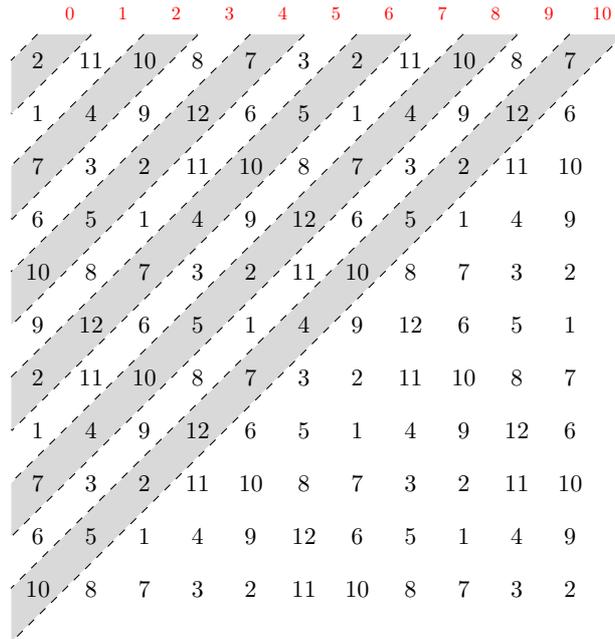
\begin{equation*}
\begin{tikzpicture}[-,thick,scale=1]
\def \wiggle {0.7};

\node[scale=0.7,color=red] at (0.6*\wiggle,0.9*\wiggle) {$0$};
\node[scale=0.7,color=red] at (1.6*\wiggle,0.9*\wiggle) {$1$};
\node[scale=0.7,color=red] at (2.6*\wiggle,0.9*\wiggle) {$2$};
\node[scale=0.7,color=red] at (3.6*\wiggle,0.9*\wiggle) {$3$};
\node[scale=0.7,color=red] at (4.6*\wiggle,0.9*\wiggle) {$4$};
\node[scale=0.7,color=red] at (5.6*\wiggle,0.9*\wiggle) {$5$};
\node[scale=0.7,color=red] at (6.6*\wiggle,0.9*\wiggle) {$6$};
\node[scale=0.7,color=red] at (7.6*\wiggle,0.9*\wiggle) {$7$};
\node[scale=0.7,color=red] at (8.6*\wiggle,0.9*\wiggle) {$8$};
\node[scale=0.7,color=red] at (9.6*\wiggle,0.9*\wiggle) {$9$};
\node[scale=0.7,color=red] at (10.6*\wiggle,0.9*\wiggle) {$10$};

\coordinate (98) at (0*\wiggle,0.5*\wiggle);
\coordinate (99) at (-0.5*\wiggle,-0*\wiggle);
\draw[dashed] (98) -- (99);

\coordinate (100) at (1*\wiggle,0.5*\wiggle);
\coordinate (101) at (-0.5*\wiggle,-1*\wiggle);
\draw[dashed] (100) -- (101);

\coordinate (102) at (2*\wiggle,0.5*\wiggle);
\coordinate (103) at (-0.5*\wiggle,-2*\wiggle);
\draw[dashed] (102) -- (103);

\coordinate (104) at (3*\wiggle,0.5*\wiggle);
\coordinate (105) at (-0.5*\wiggle,-3*\wiggle);
\draw[dashed] (104) -- (105);

\coordinate (106) at (4*\wiggle,0.5*\wiggle);
\coordinate (107) at (-0.5*\wiggle,-4*\wiggle);
\draw[dashed] (106) -- (107);

\coordinate (108) at (5*\wiggle,0.5*\wiggle);
\coordinate (109) at (-0.5*\wiggle,-5*\wiggle);
\draw[dashed] (108) -- (109);

\coordinate (110) at (6*\wiggle,0.5*\wiggle);
\coordinate (111) at (-0.5*\wiggle,-6*\wiggle);
\draw[dashed] (110) -- (111);

\coordinate (112) at (7*\wiggle,0.5*\wiggle);
\coordinate (113) at (-0.5*\wiggle,-7*\wiggle);
\draw[dashed] (112) -- (113);

\coordinate (114) at (8*\wiggle,0.5*\wiggle);
\coordinate (115) at (-0.5*\wiggle,-8*\wiggle);
\draw[dashed] (114) -- (115);

\coordinate (116) at (9*\wiggle,0.5*\wiggle);
\coordinate (117) at (-0.5*\wiggle,-9*\wiggle);
\draw[dashed] (116) -- (117);

\coordinate (118) at (10*\wiggle,0.5*\wiggle);
\coordinate (119) at (-0.5*\wiggle,-10*\wiggle);
\draw[dashed] (118) -- (119);

\coordinate (120) at (11*\wiggle,0.5*\wiggle);
\coordinate (121) at (-0.5*\wiggle,-11*\wiggle);
\draw[dashed] (120) -- (121);

%\coordinate (122) at (12*\wiggle,0.5*\wiggle);
%\coordinate (123) at (-0.5*\wiggle,-12*\wiggle);
%\draw[dashed] (122) -- (123);

%\path[fill=mygray] (120) to (121) to (123) to (122) to (120);

% \path[fill=mygray] (98) to (99) to (-0.5*\wiggle,0.5*\wiggle) to (98);
% \path[fill=mygray] (100) to (101) to (103) to (102) to (100);
% \path[fill=mygray] (104) to (105) to (107) to (106) to (104);
% \path[fill=mygray] (108) to (109) to (111) to (110) to (108);
% \path[fill=mygray] (112) to (113) to (115) to (114) to (112);
% \path[fill=mygray] (116) to (117) to (119) to (118) to (116);

\path[fill=mygray] (98) to (99) to (101) to (100) to (98);
\path[fill=mygray] (102) to (103) to (105) to (104) to (102);
\path[fill=mygray] (106) to (107) to (109) to (108) to (106);
\path[fill=mygray] (110) to (111) to (113) to (112) to (110);
\path[fill=mygray] (114) to (115) to (117) to (116) to (114);
\path[fill=mygray] (118) to (119) to (121) to (120) to (118);

\node[scale=0.901] at (0*\wiggle,0*\wiggle) {$2$};
\node[scale=0.901] at (1*\wiggle,0*\wiggle) {$11$};
\node[scale=0.901] at (2*\wiggle,0*\wiggle) {$10$};
\node[scale=0.901] at (3*\wiggle,0*\wiggle) {$8$};
\node[scale=0.901] at (4*\wiggle,0*\wiggle) {$7$};
\node[scale=0.901] at (5*\wiggle,0*\wiggle) {$3$};
\node[scale=0.901] at (6*\wiggle,0*\wiggle) {$2$};
\node[scale=0.901] at (7*\wiggle,0*\wiggle) {$11$};
\node[scale=0.901] at (8*\wiggle,0*\wiggle) {$10$};
\node[scale=0.901] at (9*\wiggle,0*\wiggle) {$8$};
\node[scale=0.901] at (10*\wiggle,0*\wiggle) {$7$};
%\node[scale=0.901] at (11*\wiggle,0*\wiggle) {$3$};

\node[scale=0.901] at (0*\wiggle,-1*\wiggle) {$1$};
\node[scale=0.901] at (1*\wiggle,-1*\wiggle) {$4$};
\node[scale=0.901] at (2*\wiggle,-1*\wiggle) {$9$};
\node[scale=0.901] at (3*\wiggle,-1*\wiggle) {$12$};
\node[scale=0.901] at (4*\wiggle,-1*\wiggle) {$6$};
\node[scale=0.901] at (5*\wiggle,-1*\wiggle) {$5$};
\node[scale=0.901] at (6*\wiggle,-1*\wiggle) {$1$};
\node[scale=0.901] at (7*\wiggle,-1*\wiggle) {$4$};
\node[scale=0.901] at (8*\wiggle,-1*\wiggle) {$9$};
\node[scale=0.901] at (9*\wiggle,-1*\wiggle) {$12$};
\node[scale=0.901] at (10*\wiggle,-1*\wiggle) {$6$};
%\node[scale=0.901] at (11*\wiggle,-1*\wiggle) {$5$};

\node[scale=0.901] at (0*\wiggle,-2*\wiggle) {$7$};
\node[scale=0.901] at (1*\wiggle,-2*\wiggle) {$3$};
\node[scale=0.901] at (2*\wiggle,-2*\wiggle) {$2$};
\node[scale=0.901] at (3*\wiggle,-2*\wiggle) {$11$};
\node[scale=0.901] at (4*\wiggle,-2*\wiggle) {$10$};
\node[scale=0.901] at (5*\wiggle,-2*\wiggle) {$8$};
\node[scale=0.901] at (6*\wiggle,-2*\wiggle) {$7$};
\node[scale=0.901] at (7*\wiggle,-2*\wiggle) {$3$};
\node[scale=0.901] at (8*\wiggle,-2*\wiggle) {$2$};
\node[scale=0.901] at (9*\wiggle,-2*\wiggle) {$11$};
\node[scale=0.901] at (10*\wiggle,-2*\wiggle) {$10$};
%\node[scale=0.901] at (11*\wiggle,-2*\wiggle) {$8$};

\node[scale=0.901] at (0*\wiggle,-3*\wiggle) {$6$};
\node[scale=0.901] at (1*\wiggle,-3*\wiggle) {$5$};
\node[scale=0.901] at (2*\wiggle,-3*\wiggle) {$1$};
\node[scale=0.901] at (3*\wiggle,-3*\wiggle) {$4$};
\node[scale=0.901] at (4*\wiggle,-3*\wiggle) {$9$};
\node[scale=0.901] at (5*\wiggle,-3*\wiggle) {$12$};
\node[scale=0.901] at (6*\wiggle,-3*\wiggle) {$6$};
\node[scale=0.901] at (7*\wiggle,-3*\wiggle) {$5$};
\node[scale=0.901] at (8*\wiggle,-3*\wiggle) {$1$};
\node[scale=0.901] at (9*\wiggle,-3*\wiggle) {$4$};
\node[scale=0.901] at (10*\wiggle,-3*\wiggle) {$9$};
%\node[scale=0.901] at (11*\wiggle,-3*\wiggle) {$12$};

\node[scale=0.901] at (0*\wiggle,-4*\wiggle) {$10$};
\node[scale=0.901] at (1*\wiggle,-4*\wiggle) {$8$};
\node[scale=0.901] at (2*\wiggle,-4*\wiggle) {$7$};
\node[scale=0.901] at (3*\wiggle,-4*\wiggle) {$3$};
\node[scale=0.901] at (4*\wiggle,-4*\wiggle) {$2$};
\node[scale=0.901] at (5*\wiggle,-4*\wiggle) {$11$};
\node[scale=0.901] at (6*\wiggle,-4*\wiggle) {$10$};
\node[scale=0.901] at (7*\wiggle,-4*\wiggle) {$8$};
\node[scale=0.901] at (8*\wiggle,-4*\wiggle) {$7$};
\node[scale=0.901] at (9*\wiggle,-4*\wiggle) {$3$};
\node[scale=0.901] at (10*\wiggle,-4*\wiggle) {$2$};
%\node[scale=0.901] at (11*\wiggle,-4*\wiggle) {$11$};

\node[scale=0.901] at (0*\wiggle,-5*\wiggle) {$9$};
\node[scale=0.901] at (1*\wiggle,-5*\wiggle) {$12$};
\node[scale=0.901] at (2*\wiggle,-5*\wiggle) {$6$};
\node[scale=0.901] at (3*\wiggle,-5*\wiggle) {$5$};
\node[scale=0.901] at (4*\wiggle,-5*\wiggle) {$1$};
\node[scale=0.901] at (5*\wiggle,-5*\wiggle) {$4$};
\node[scale=0.901] at (6*\wiggle,-5*\wiggle) {$9$};
\node[scale=0.901] at (7*\wiggle,-5*\wiggle) {$12$};
\node[scale=0.901] at (8*\wiggle,-5*\wiggle) {$6$};
\node[scale=0.901] at (9*\wiggle,-5*\wiggle) {$5$};
\node[scale=0.901] at (10*\wiggle,-5*\wiggle) {$1$};
%\node[scale=0.901] at (11*\wiggle,-5*\wiggle) {$4$};

\node[scale=0.901] at (0*\wiggle,-6*\wiggle) {$2$};
\node[scale=0.901] at (1*\wiggle,-6*\wiggle) {$11$};
\node[scale=0.901] at (2*\wiggle,-6*\wiggle) {$10$};
\node[scale=0.901] at (3*\wiggle,-6*\wiggle) {$8$};
\node[scale=0.901] at (4*\wiggle,-6*\wiggle) {$7$};
\node[scale=0.901] at (5*\wiggle,-6*\wiggle) {$3$};
\node[scale=0.901] at (6*\wiggle,-6*\wiggle) {$2$};
\node[scale=0.901] at (7*\wiggle,-6*\wiggle) {$11$};
\node[scale=0.901] at (8*\wiggle,-6*\wiggle) {$10$};
\node[scale=0.901] at (9*\wiggle,-6*\wiggle) {$8$};
\node[scale=0.901] at (10*\wiggle,-6*\wiggle) {$7$};
%\node[scale=0.901] at (11*\wiggle,-6*\wiggle) {$3$};

\node[scale=0.901] at (0*\wiggle,-7*\wiggle) {$1$};
\node[scale=0.901] at (1*\wiggle,-7*\wiggle) {$4$};
\node[scale=0.901] at (2*\wiggle,-7*\wiggle) {$9$};
\node[scale=0.901] at (3*\wiggle,-7*\wiggle) {$12$};
\node[scale=0.901] at (4*\wiggle,-7*\wiggle) {$6$};
\node[scale=0.901] at (5*\wiggle,-7*\wiggle) {$5$};
\node[scale=0.901] at (6*\wiggle,-7*\wiggle) {$1$};
\node[scale=0.901] at (7*\wiggle,-7*\wiggle) {$4$};
\node[scale=0.901] at (8*\wiggle,-7*\wiggle) {$9$};
\node[scale=0.901] at (9*\wiggle,-7*\wiggle) {$12$};
\node[scale=0.901] at (10*\wiggle,-7*\wiggle) {$6$};
%\node[scale=0.901] at (11*\wiggle,-7*\wiggle) {$5$};

\node[scale=0.901] at (0*\wiggle,-8*\wiggle) {$7$};
\node[scale=0.901] at (1*\wiggle,-8*\wiggle) {$3$};
\node[scale=0.901] at (2*\wiggle,-8*\wiggle) {$2$};
\node[scale=0.901] at (3*\wiggle,-8*\wiggle) {$11$};
\node[scale=0.901] at (4*\wiggle,-8*\wiggle) {$10$};
\node[scale=0.901] at (5*\wiggle,-8*\wiggle) {$8$};
\node[scale=0.901] at (6*\wiggle,-8*\wiggle) {$7$};
\node[scale=0.901] at (7*\wiggle,-8*\wiggle) {$3$};
\node[scale=0.901] at (8*\wiggle,-8*\wiggle) {$2$};
\node[scale=0.901] at (9*\wiggle,-8*\wiggle) {$11$};
\node[scale=0.901] at (10*\wiggle,-8*\wiggle) {$10$};
%\node[scale=0.901] at (11*\wiggle,-8*\wiggle) {$8$};

\node[scale=0.901] at (0*\wiggle,-9*\wiggle) {$6$};
\node[scale=0.901] at (1*\wiggle,-9*\wiggle) {$5$};
\node[scale=0.901] at (2*\wiggle,-9*\wiggle) {$1$};
\node[scale=0.901] at (3*\wiggle,-9*\wiggle) {$4$};
\node[scale=0.901] at (4*\wiggle,-9*\wiggle) {$9$};
\node[scale=0.901] at (5*\wiggle,-9*\wiggle) {$12$};
\node[scale=0.901] at (6*\wiggle,-9*\wiggle) {$6$};
\node[scale=0.901] at (7*\wiggle,-9*\wiggle) {$5$};
\node[scale=0.901] at (8*\wiggle,-9*\wiggle) {$1$};
\node[scale=0.901] at (9*\wiggle,-9*\wiggle) {$4$};
\node[scale=0.901] at (10*\wiggle,-9*\wiggle) {$9$};
%\node[scale=0.901] at (11*\wiggle,-9*\wiggle) {$12$};

\node[scale=0.901] at (0*\wiggle,-10*\wiggle) {$10$};
\node[scale=0.901] at (1*\wiggle,-10*\wiggle) {$8$};
\node[scale=0.901] at (2*\wiggle,-10*\wiggle) {$7$};
\node[scale=0.901] at (3*\wiggle,-10*\wiggle) {$3$};
\node[scale=0.901] at (4*\wiggle,-10*\wiggle) {$2$};
\node[scale=0.901] at (5*\wiggle,-10*\wiggle) {$11$};
\node[scale=0.901] at (6*\wiggle,-10*\wiggle) {$10$};
\node[scale=0.901] at (7*\wiggle,-10*\wiggle) {$8$};
\node[scale=0.901] at (8*\wiggle,-10*\wiggle) {$7$};
\node[scale=0.901] at (9*\wiggle,-10*\wiggle) {$3$};
\node[scale=0.901] at (10*\wiggle,-10*\wiggle) {$2$};
%\node[scale=0.901] at (11*\wiggle,-10*\wiggle) {$11$};

% \node[scale=0.901] at (0*\wiggle,-11*\wiggle) {$9$};
% \node[scale=0.901] at (1*\wiggle,-11*\wiggle) {$12$};
% \node[scale=0.901] at (2*\wiggle,-11*\wiggle) {$6$};
% \node[scale=0.901] at (3*\wiggle,-11*\wiggle) {$5$};
% \node[scale=0.901] at (4*\wiggle,-11*\wiggle) {$1$};
% \node[scale=0.901] at (5*\wiggle,-11*\wiggle) {$4$};
% \node[scale=0.901] at (6*\wiggle,-11*\wiggle) {$9$};
% \node[scale=0.901] at (7*\wiggle,-11*\wiggle) {$12$};
% \node[scale=0.901] at (8*\wiggle,-11*\wiggle) {$6$};
% \node[scale=0.901] at (9*\wiggle,-11*\wiggle) {$5$};
% \node[scale=0.901] at (10*\wiggle,-11*\wiggle) {$1$};
% \node[scale=0.901] at (11*\wiggle,-11*\wiggle) {$4$};

\end{tikzpicture}
\end{equation*}

\noindent This is simply a portion of the quiver $\mathbb{ZA}_\infty^\infty$, where we have removed the arrows to improve clarity. Since we are interested in the paths starting from vertex $2$, we have considered a portion in which this vertex is in the top left corner. Each vertex appearing in the grid corresponds to a unique path in $\Lambda$ from vertex $2$ to vertex $j$. Therefore, the number of paths of length $\ell$ from vertex $2$ to vertex $j$ (equivalently, $\dim_\Bbbk (e_2 \Lambda e_j)_\ell$) is equal to the number of occurrences of $j$ along the diagonal strip indexed by the red non-negative integer $\ell$. For example, the following table records the number of paths in $\Lambda$ from vertex $2$ to the vertices $j=1,2,3,4$ of length at most $10$:
\begin{align*}
\begin{tabular}{c|c c c c c c c c c c c}
\multirow{2}{*}{\makecell{Terminal\\vertex}} & \multicolumn{11}{c}{Path length} \\ \cline{2-12}
  & 0 & 1 & 2 & 3 & 4 & 5 & 6 & 7 & 8 & 9 & 10 \\ \hline
1 & 0 & 1 & 0 & 0 & 0 & 1 & 0 & 2 & 0 & 1 & 0 \\ 
2 & 1 & 0 & 0 & 0 & 1 & 0 & 2 & 0 & 1 & 0 & 2 \\ 
3 & 0 & 0 & 0 & 1 & 0 & 1 & 0 & 1 & 0 & 2 & 0 \\ 
4 & 0 & 0 & 1 & 0 & 0 & 0 & 1 & 0 & 1 & 0 & 1 
\end{tabular}
\end{align*}

\indent We can modify the above approach to also allow us to calculate the number of paths from a vertex $i$ to a vertex $j$ which cannot, in $\Lambda$, be factored through a certain vertex (or vertices). This will be relevant in the next section. \\
\indent For example, suppose we are interested in counting the number of paths in $\Lambda$ which start at vertex $2$ and which do not pass through vertex $7$, say. Equivalently, we are interested in calculating the dimension of the vector space $e_2 \Lambda / e_2 \Lambda e_7 \Lambda$. Using the lattice representation of $\Lambda$, this happens if and only if that vertex does not lie directly to the south or directly to the east of one of the instances of vertex $7$. Below, the vertices for which there exists a nonzero path to it from vertex $2$ in $\Lambda$ are shown in black, and the other vertices are shown in grey:
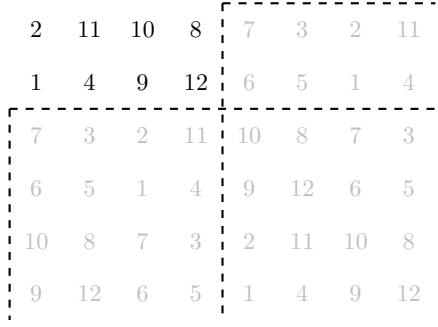
\begin{equation*}
\begin{tikzpicture}[-,thick,scale=1]
\def \wiggle {0.7};

\draw[dashed] (-0.5*\wiggle,-1.5*\wiggle) -- (7.5*\wiggle,-1.5*\wiggle);
\draw[dashed] (-0.5*\wiggle,-1.5*\wiggle) -- (-0.5*\wiggle,-5.5*\wiggle);
\draw[dashed] (3.5*\wiggle,0.5*\wiggle) -- (7.5*\wiggle,0.5*\wiggle);
\draw[dashed] (3.5*\wiggle,0.5*\wiggle) -- (3.5*\wiggle,-5.5*\wiggle);

\node[scale=0.901] at (0*\wiggle,0*\wiggle) {$2$};
\node[scale=0.901] at (1*\wiggle,0*\wiggle) {$11$};
\node[scale=0.901] at (2*\wiggle,0*\wiggle) {$10$};
\node[scale=0.901] at (3*\wiggle,0*\wiggle) {$8$};
\node[scale=0.901,color=lightgray] at (4*\wiggle,0*\wiggle) {$7$};
\node[scale=0.901,color=lightgray] at (5*\wiggle,0*\wiggle) {$3$};
\node[scale=0.901,color=lightgray] at (6*\wiggle,0*\wiggle) {$2$};
\node[scale=0.901,color=lightgray] at (7*\wiggle,0*\wiggle) {$11$};

\node[scale=0.901] at (0*\wiggle,-1*\wiggle) {$1$};
\node[scale=0.901] at (1*\wiggle,-1*\wiggle) {$4$};
\node[scale=0.901] at (2*\wiggle,-1*\wiggle) {$9$};
\node[scale=0.901] at (3*\wiggle,-1*\wiggle) {$12$};
\node[scale=0.901,color=lightgray] at (4*\wiggle,-1*\wiggle) {$6$};
\node[scale=0.901,color=lightgray] at (5*\wiggle,-1*\wiggle) {$5$};
\node[scale=0.901,color=lightgray] at (6*\wiggle,-1*\wiggle) {$1$};
\node[scale=0.901,color=lightgray] at (7*\wiggle,-1*\wiggle) {$4$};

\node[scale=0.901,color=lightgray] at (0*\wiggle,-2*\wiggle) {$7$};
\node[scale=0.901,color=lightgray] at (1*\wiggle,-2*\wiggle) {$3$};
\node[scale=0.901,color=lightgray] at (2*\wiggle,-2*\wiggle) {$2$};
\node[scale=0.901,color=lightgray] at (3*\wiggle,-2*\wiggle) {$11$};
\node[scale=0.901,color=lightgray] at (4*\wiggle,-2*\wiggle) {$10$};
\node[scale=0.901,color=lightgray] at (5*\wiggle,-2*\wiggle) {$8$};
\node[scale=0.901,color=lightgray] at (6*\wiggle,-2*\wiggle) {$7$};
\node[scale=0.901,color=lightgray] at (7*\wiggle,-2*\wiggle) {$3$};

\node[scale=0.901,color=lightgray] at (0*\wiggle,-3*\wiggle) {$6$};
\node[scale=0.901,color=lightgray] at (1*\wiggle,-3*\wiggle) {$5$};
\node[scale=0.901,color=lightgray] at (2*\wiggle,-3*\wiggle) {$1$};
\node[scale=0.901,color=lightgray] at (3*\wiggle,-3*\wiggle) {$4$};
\node[scale=0.901,color=lightgray] at (4*\wiggle,-3*\wiggle) {$9$};
\node[scale=0.901,color=lightgray] at (5*\wiggle,-3*\wiggle) {$12$};
\node[scale=0.901,color=lightgray] at (6*\wiggle,-3*\wiggle) {$6$};
\node[scale=0.901,color=lightgray] at (7*\wiggle,-3*\wiggle) {$5$};

\node[scale=0.901,color=lightgray] at (0*\wiggle,-4*\wiggle) {$10$};
\node[scale=0.901,color=lightgray] at (1*\wiggle,-4*\wiggle) {$8$};
\node[scale=0.901,color=lightgray] at (2*\wiggle,-4*\wiggle) {$7$};
\node[scale=0.901,color=lightgray] at (3*\wiggle,-4*\wiggle) {$3$};
\node[scale=0.901,color=lightgray] at (4*\wiggle,-4*\wiggle) {$2$};
\node[scale=0.901,color=lightgray] at (5*\wiggle,-4*\wiggle) {$11$};
\node[scale=0.901,color=lightgray] at (6*\wiggle,-4*\wiggle) {$10$};
\node[scale=0.901,color=lightgray] at (7*\wiggle,-4*\wiggle) {$8$};

\node[scale=0.901,color=lightgray] at (0*\wiggle,-5*\wiggle) {$9$};
\node[scale=0.901,color=lightgray] at (1*\wiggle,-5*\wiggle) {$12$};
\node[scale=0.901,color=lightgray] at (2*\wiggle,-5*\wiggle) {$6$};
\node[scale=0.901,color=lightgray] at (3*\wiggle,-5*\wiggle) {$5$};
\node[scale=0.901,color=lightgray] at (4*\wiggle,-5*\wiggle) {$1$};
\node[scale=0.901,color=lightgray] at (5*\wiggle,-5*\wiggle) {$4$};
\node[scale=0.901,color=lightgray] at (6*\wiggle,-5*\wiggle) {$9$};
\node[scale=0.901,color=lightgray] at (7*\wiggle,-5*\wiggle) {$12$};

\end{tikzpicture}
\end{equation*}
\noindent In particular, $\dim_\Bbbk e_2 \Lambda / e_2 \Lambda e_7 \Lambda = 8$. If we are instead interested in paths which cannot be factored through the vertex $9$, the picture we obtain is as follows:
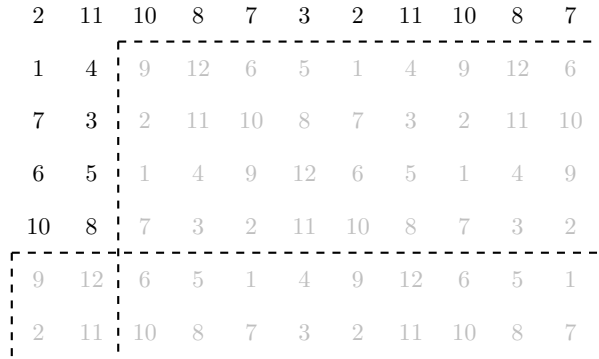
\begin{equation*}
\begin{tikzpicture}[-,thick,scale=1]
\def \wiggle {0.7};

\draw[dashed] (-0.5*\wiggle,-4.5*\wiggle) -- (10.5*\wiggle,-4.5*\wiggle);
\draw[dashed] (-0.5*\wiggle,-4.5*\wiggle) -- (-0.5*\wiggle,-6.5*\wiggle);
\draw[dashed] (1.5*\wiggle,-0.5*\wiggle) -- (10.5*\wiggle,-0.5*\wiggle);
\draw[dashed] (1.5*\wiggle,-0.5*\wiggle) -- (1.5*\wiggle,-6.5*\wiggle);

\node[scale=0.901] at (0*\wiggle,0*\wiggle) {$2$};
\node[scale=0.901] at (1*\wiggle,0*\wiggle) {$11$};
\node[scale=0.901] at (2*\wiggle,0*\wiggle) {$10$};
\node[scale=0.901] at (3*\wiggle,0*\wiggle) {$8$};
\node[scale=0.901] at (4*\wiggle,0*\wiggle) {$7$};
\node[scale=0.901] at (5*\wiggle,0*\wiggle) {$3$};
\node[scale=0.901] at (6*\wiggle,0*\wiggle) {$2$};
\node[scale=0.901] at (7*\wiggle,0*\wiggle) {$11$};
\node[scale=0.901] at (8*\wiggle,0*\wiggle) {$10$};
\node[scale=0.901] at (9*\wiggle,0*\wiggle) {$8$};
\node[scale=0.901] at (10*\wiggle,0*\wiggle) {$7$};
%\node[scale=0.901] at (11*\wiggle,0*\wiggle) {$3$};

\node[scale=0.901] at (0*\wiggle,-1*\wiggle) {$1$};
\node[scale=0.901] at (1*\wiggle,-1*\wiggle) {$4$};
\node[scale=0.901,color=lightgray] at (2*\wiggle,-1*\wiggle) {$9$};
\node[scale=0.901,color=lightgray] at (3*\wiggle,-1*\wiggle) {$12$};
\node[scale=0.901,color=lightgray] at (4*\wiggle,-1*\wiggle) {$6$};
\node[scale=0.901,color=lightgray] at (5*\wiggle,-1*\wiggle) {$5$};
\node[scale=0.901,color=lightgray] at (6*\wiggle,-1*\wiggle) {$1$};
\node[scale=0.901,color=lightgray] at (7*\wiggle,-1*\wiggle) {$4$};
\node[scale=0.901,color=lightgray] at (8*\wiggle,-1*\wiggle) {$9$};
\node[scale=0.901,color=lightgray] at (9*\wiggle,-1*\wiggle) {$12$};
\node[scale=0.901,color=lightgray] at (10*\wiggle,-1*\wiggle) {$6$};
%\node[scale=0.901] at (11*\wiggle,-1*\wiggle) {$5$};

\node[scale=0.901] at (0*\wiggle,-2*\wiggle) {$7$};
\node[scale=0.901] at (1*\wiggle,-2*\wiggle) {$3$};
\node[scale=0.901,color=lightgray] at (2*\wiggle,-2*\wiggle) {$2$};
\node[scale=0.901,color=lightgray] at (3*\wiggle,-2*\wiggle) {$11$};
\node[scale=0.901,color=lightgray] at (4*\wiggle,-2*\wiggle) {$10$};
\node[scale=0.901,color=lightgray] at (5*\wiggle,-2*\wiggle) {$8$};
\node[scale=0.901,color=lightgray] at (6*\wiggle,-2*\wiggle) {$7$};
\node[scale=0.901,color=lightgray] at (7*\wiggle,-2*\wiggle) {$3$};
\node[scale=0.901,color=lightgray] at (8*\wiggle,-2*\wiggle) {$2$};
\node[scale=0.901,color=lightgray] at (9*\wiggle,-2*\wiggle) {$11$};
\node[scale=0.901,color=lightgray] at (10*\wiggle,-2*\wiggle) {$10$};
%\node[scale=0.901,color=lightgray] at (11*\wiggle,-2*\wiggle) {$8$};

\node[scale=0.901] at (0*\wiggle,-3*\wiggle) {$6$};
\node[scale=0.901] at (1*\wiggle,-3*\wiggle) {$5$};
\node[scale=0.901,color=lightgray] at (2*\wiggle,-3*\wiggle) {$1$};
\node[scale=0.901,color=lightgray] at (3*\wiggle,-3*\wiggle) {$4$};
\node[scale=0.901,color=lightgray] at (4*\wiggle,-3*\wiggle) {$9$};
\node[scale=0.901,color=lightgray] at (5*\wiggle,-3*\wiggle) {$12$};
\node[scale=0.901,color=lightgray] at (6*\wiggle,-3*\wiggle) {$6$};
\node[scale=0.901,color=lightgray] at (7*\wiggle,-3*\wiggle) {$5$};
\node[scale=0.901,color=lightgray] at (8*\wiggle,-3*\wiggle) {$1$};
\node[scale=0.901,color=lightgray] at (9*\wiggle,-3*\wiggle) {$4$};
\node[scale=0.901,color=lightgray] at (10*\wiggle,-3*\wiggle) {$9$};
%\node[scale=0.901] at (11*\wiggle,-3*\wiggle) {$12$};

\node[scale=0.901] at (0*\wiggle,-4*\wiggle) {$10$};
\node[scale=0.901] at (1*\wiggle,-4*\wiggle) {$8$};
\node[scale=0.901,color=lightgray] at (2*\wiggle,-4*\wiggle) {$7$};
\node[scale=0.901,color=lightgray] at (3*\wiggle,-4*\wiggle) {$3$};
\node[scale=0.901,color=lightgray] at (4*\wiggle,-4*\wiggle) {$2$};
\node[scale=0.901,color=lightgray] at (5*\wiggle,-4*\wiggle) {$11$};
\node[scale=0.901,color=lightgray] at (6*\wiggle,-4*\wiggle) {$10$};
\node[scale=0.901,color=lightgray] at (7*\wiggle,-4*\wiggle) {$8$};
\node[scale=0.901,color=lightgray] at (8*\wiggle,-4*\wiggle) {$7$};
\node[scale=0.901,color=lightgray] at (9*\wiggle,-4*\wiggle) {$3$};
\node[scale=0.901,color=lightgray] at (10*\wiggle,-4*\wiggle) {$2$};
%\node[scale=0.901,color=lightgray] at (11*\wiggle,-4*\wiggle) {$11$};

\node[scale=0.901,color=lightgray] at (0*\wiggle,-5*\wiggle) {$9$};
\node[scale=0.901,color=lightgray] at (1*\wiggle,-5*\wiggle) {$12$};
\node[scale=0.901,color=lightgray] at (2*\wiggle,-5*\wiggle) {$6$};
\node[scale=0.901,color=lightgray] at (3*\wiggle,-5*\wiggle) {$5$};
\node[scale=0.901,color=lightgray] at (4*\wiggle,-5*\wiggle) {$1$};
\node[scale=0.901,color=lightgray] at (5*\wiggle,-5*\wiggle) {$4$};
\node[scale=0.901,color=lightgray] at (6*\wiggle,-5*\wiggle) {$9$};
\node[scale=0.901,color=lightgray] at (7*\wiggle,-5*\wiggle) {$12$};
\node[scale=0.901,color=lightgray] at (8*\wiggle,-5*\wiggle) {$6$};
\node[scale=0.901,color=lightgray] at (9*\wiggle,-5*\wiggle) {$5$};
\node[scale=0.901,color=lightgray] at (10*\wiggle,-5*\wiggle) {$1$};
%\node[scale=0.901,color=lightgray] at (11*\wiggle,-5*\wiggle) {$4$};

\node[scale=0.901,color=lightgray] at (0*\wiggle,-6*\wiggle) {$2$};
\node[scale=0.901,color=lightgray] at (1*\wiggle,-6*\wiggle) {$11$};
\node[scale=0.901,color=lightgray] at (2*\wiggle,-6*\wiggle) {$10$};
\node[scale=0.901,color=lightgray] at (3*\wiggle,-6*\wiggle) {$8$};
\node[scale=0.901,color=lightgray] at (4*\wiggle,-6*\wiggle) {$7$};
\node[scale=0.901,color=lightgray] at (5*\wiggle,-6*\wiggle) {$3$};
\node[scale=0.901,color=lightgray] at (6*\wiggle,-6*\wiggle) {$2$};
\node[scale=0.901,color=lightgray] at (7*\wiggle,-6*\wiggle) {$11$};
\node[scale=0.901,color=lightgray] at (8*\wiggle,-6*\wiggle) {$10$};
\node[scale=0.901,color=lightgray] at (9*\wiggle,-6*\wiggle) {$8$};
\node[scale=0.901,color=lightgray] at (10*\wiggle,-6*\wiggle) {$7$};
%\node[scale=0.901] at (11*\wiggle,-6*\wiggle) {$3$};

\end{tikzpicture}
\end{equation*}
\noindent Due to the periodicity of the top row, there are infinitely many nonzero paths in $\Lambda$ which start at the vertex $2$ and do not pass through the vertex $9$. Therefore $\dim_\Bbbk e_2 \Lambda / e_2 \Lambda e_9 \Lambda = \infty$.

%%%%%%%%%%%%%%%%%%%%%%%%%%%%%%%%%%%%%%%%%%%%%%%%%%%%
%%%%%%%%%%%%%%%%%%%%%%%%%%%%%%%%%%%%%%%%%%%%%%%%%%%%
%%%%%%%%%%%%%%%%%%%%%%%%%%%%%%%%%%%%%%%%%%%%%%%%%%%%

\section{Proofs of Main Results}
\subsection{On when the Auslander map is an isomorphism}
We now work towards providing necessary and sufficient conditions for the Auslander map corresponding to the pairs $(A,G)$ from Proposition \ref{prop:Classifying} to be an isomorphism. First, we require an elementary lemma:

\begin{lem} \label{EasyLem}
Suppose that $G$ is a finite group and $g,h \in G$ do not commute. Then the elements
\begin{align*}
1_G, g, gh, ghg, ghgh, \dots, (gh)^{k-1} , (gh)^{k-1} g 
\end{align*}
are all distinct if and only if $|gh| \geqslant k$. In particular, the above list consists of every element of $G$ without repetitions if and only if $|gh| = \frac{1}{2}|G|$.
\end{lem}
\begin{proof}
Necessity is clear, so we prove only sufficiency. Suppose we have a repetition. If it is of the form $(gh)^m g^\varepsilon = (gh)^n g^\varepsilon$, where $0 \leqslant m < n \leqslant k-1$ and $\varepsilon \in \{0,1\}$, it follows that $|gh| \leqslant n-m \leqslant k-1$. \\ 
\indent Alternatively, a repetition could have the form $(gh)^m = (gh)^n g$, where $0 \leqslant m < n \leqslant k-1$. Then $g = (gh)^{m-n}$, but also
\begin{align*}
g (hg)^{n-m} = (gh)^{n-m} g = 1,
\end{align*}
so that $g^{-1} = (hg)^{n-m}$. Hence
\begin{align*}
hg^{-1} = h (gh)^{n-m} = (hg)^{n-m} h = g^{-1}h.
\end{align*}
Therefore $gh=hg$, contrary to our assumption. The case where $m > n$ is similar, exhausting all of the possibilities.
\end{proof}

With this lemma in hand, we are able to exploit the structure of the algebra $\Lambda$ to determine precisely when the Auslander map is an isomorphism:

\begin{thm} \label{AuslanderForCoaction}
Suppose that $A$ and $G$ are as in Proposition \ref{prop:Classifying}. Then the Auslander map for the pair $(A,G)$ is an isomorphism if and only if $|ab| = \frac{1}{2} |G|$.
\end{thm}
\begin{proof}
%The Auslander map is an isomorphism iff $\Lambda / \Lambda e_0 \Lambda$ is finite-dimensional, iff $e_i \Lambda / e_i \Lambda e_0 \Lambda$ is finite-dimensional for each vertex $j$, iff every row and column of the toroidal representation of $\Lambda$ contains every element of $G$, iff for every $g \in G$, the list $g, ga, gab, gaba, gabab, \dots$ includes every element of $G$, iff $ab$ has order $\frac{1}{2} |G|$.
We have the following sequence of equivalent statements:
\begin{align*}
\text{The Auslander map is an}&\text{ isomorphism} \\
&\Leftrightarrow \Lambda / \Lambda e_1 \Lambda \text{ is finite-dimensional} \tag{Theorem \ref{thm:FinDimLambda}} \\
&\Leftrightarrow e_j \Lambda / e_j \Lambda e_1 \Lambda \text{ is finite-dimensional for } \\ 
&\mathrel{\phantom{\Leftrightarrow}} \text{each vertex } j \\
&\Leftrightarrow \text{every row and column of the toroidal} \\
&\mathrel{\phantom{\Leftrightarrow}} \text{presentation of } \Lambda \text{ contains } 1_G \tag{Section \ref{sec:DimensionCounting}}\\
&\Leftrightarrow \text{every row and column of the toroidal} \\
&\mathrel{\phantom{\Leftrightarrow}} \text{presentation of } \Lambda \text{ contains every element of } G \\
&\Leftrightarrow \text{for all } g \in G, \text{ the list } g, a^{-1}g, b^{-1} a^{-1} g, a^{-1} b^{-1} a^{-1} g, \dots \\
&\mathrel{\phantom{\Leftrightarrow}} \text{includes every element of } G \\
&\Leftrightarrow ab \text{ has order } \tfrac{1}{2}|G|, \tag{Lemma \ref{EasyLem}}
\end{align*}
where we apply Lemma \ref{EasyLem} with $g=a^{-1}$ and $h=b^{-1}$. This proves the result.
\end{proof}

\begin{example}
Suppose that
\begin{align*}
G = \langle a,b \mid a^2=b^2, a^4=b^4=(a^3b)^3=1 \rangle,
\end{align*}
which is abstractly isomorphic to the group in the running example in Section \ref{sec:Quivers}. If we give $A$ a group-grading by declaring $\deg_G u = a$ and $\deg_G v = b$, it turns out that the invariant ring $A_{1_G}$ has different properties compared with in Section \ref{sec:Quivers}. \\
\indent It is easy to see that, as a set, we have
\begin{align*}
G = \{ 1,a,a^2,a^3, b,ab,a^2b,a^3b, ba,aba,a^2ba,a^3ba \},
\end{align*}
as was the case in Section \ref{sec:Quivers}. If we number the vertices of the McKay quiver by the index of the group element in the above set, then the toroidal presentation of the McKay quiver is as follows (where we have omitted the arrows but, as always, they implicitly point directly south and east):

\begin{equation*}
\begin{tikzpicture}[->,>=stealth,thick,scale=1]
\def \wiggle {0.5};

\node[scale=0.71] at (0*\wiggle,0*\wiggle) {$1$};
\node[scale=0.71] at (1*\wiggle,0*\wiggle) {$4$};
\node[scale=0.71] at (2*\wiggle,0*\wiggle) {$9$};
\node[scale=0.71] at (3*\wiggle,0*\wiggle) {$12$};
\node[scale=0.71] at (4*\wiggle,0*\wiggle) {$8$};
\node[scale=0.71] at (5*\wiggle,0*\wiggle) {$7$};
\node[scale=0.71] at (6*\wiggle,0*\wiggle) {$3$};
\node[scale=0.71] at (7*\wiggle,0*\wiggle) {$2$};
\node[scale=0.71] at (8*\wiggle,0*\wiggle) {$11$};
\node[scale=0.71] at (9*\wiggle,0*\wiggle) {$10$};
\node[scale=0.71] at (10*\wiggle,0*\wiggle) {$6$};
\node[scale=0.71] at (11*\wiggle,0*\wiggle) {$5$};
\node[scale=0.71] at (12*\wiggle,0*\wiggle) {$1$};

\node[scale=0.71] at (0*\wiggle,-1*\wiggle) {$7$};
\node[scale=0.71] at (1*\wiggle,-1*\wiggle) {$3$};
\node[scale=0.71] at (2*\wiggle,-1*\wiggle) {$2$};
\node[scale=0.71] at (3*\wiggle,-1*\wiggle) {$11$};
\node[scale=0.71] at (4*\wiggle,-1*\wiggle) {$10$};
\node[scale=0.71] at (5*\wiggle,-1*\wiggle) {$6$};
\node[scale=0.71] at (6*\wiggle,-1*\wiggle) {$5$};
\node[scale=0.71] at (7*\wiggle,-1*\wiggle) {$1$};
\node[scale=0.71] at (8*\wiggle,-1*\wiggle) {$4$};
\node[scale=0.71] at (9*\wiggle,-1*\wiggle) {$9$};
\node[scale=0.71] at (10*\wiggle,-1*\wiggle) {$12$};
\node[scale=0.71] at (11*\wiggle,-1*\wiggle) {$8$};
\node[scale=0.71] at (12*\wiggle,-1*\wiggle) {$7$};

\node[scale=0.71] at (0*\wiggle,-2*\wiggle) {$6$};
\node[scale=0.71] at (1*\wiggle,-2*\wiggle) {$5$};
\node[scale=0.71] at (2*\wiggle,-2*\wiggle) {$1$};
\node[scale=0.71] at (3*\wiggle,-2*\wiggle) {$4$};
\node[scale=0.71] at (4*\wiggle,-2*\wiggle) {$9$};
\node[scale=0.71] at (5*\wiggle,-2*\wiggle) {$12$};
\node[scale=0.71] at (6*\wiggle,-2*\wiggle) {$8$};
\node[scale=0.71] at (7*\wiggle,-2*\wiggle) {$7$};
\node[scale=0.71] at (8*\wiggle,-2*\wiggle) {$3$};
\node[scale=0.71] at (9*\wiggle,-2*\wiggle) {$2$};
\node[scale=0.71] at (10*\wiggle,-2*\wiggle) {$11$};
\node[scale=0.71] at (11*\wiggle,-2*\wiggle) {$10$};
\node[scale=0.71] at (12*\wiggle,-2*\wiggle) {$6$};

\node[scale=0.71] at (0*\wiggle,-3*\wiggle) {$12$};
\node[scale=0.71] at (1*\wiggle,-3*\wiggle) {$8$};
\node[scale=0.71] at (2*\wiggle,-3*\wiggle) {$7$};
\node[scale=0.71] at (3*\wiggle,-3*\wiggle) {$3$};
\node[scale=0.71] at (4*\wiggle,-3*\wiggle) {$2$};
\node[scale=0.71] at (5*\wiggle,-3*\wiggle) {$11$};
\node[scale=0.71] at (6*\wiggle,-3*\wiggle) {$10$};
\node[scale=0.71] at (7*\wiggle,-3*\wiggle) {$6$};
\node[scale=0.71] at (8*\wiggle,-3*\wiggle) {$5$};
\node[scale=0.71] at (9*\wiggle,-3*\wiggle) {$1$};
\node[scale=0.71] at (10*\wiggle,-3*\wiggle) {$4$};
\node[scale=0.71] at (11*\wiggle,-3*\wiggle) {$9$};
\node[scale=0.71] at (12*\wiggle,-3*\wiggle) {$12$};

\node[scale=0.71] at (0*\wiggle,-4*\wiggle) {$11$};
\node[scale=0.71] at (1*\wiggle,-4*\wiggle) {$10$};
\node[scale=0.71] at (2*\wiggle,-4*\wiggle) {$6$};
\node[scale=0.71] at (3*\wiggle,-4*\wiggle) {$5$};
\node[scale=0.71] at (4*\wiggle,-4*\wiggle) {$1$};
\node[scale=0.71] at (5*\wiggle,-4*\wiggle) {$4$};
\node[scale=0.71] at (6*\wiggle,-4*\wiggle) {$9$};
\node[scale=0.71] at (7*\wiggle,-4*\wiggle) {$12$};
\node[scale=0.71] at (8*\wiggle,-4*\wiggle) {$8$};
\node[scale=0.71] at (9*\wiggle,-4*\wiggle) {$7$};
\node[scale=0.71] at (10*\wiggle,-4*\wiggle) {$3$};
\node[scale=0.71] at (11*\wiggle,-4*\wiggle) {$2$};
\node[scale=0.71] at (12*\wiggle,-4*\wiggle) {$11$};

\node[scale=0.71] at (0*\wiggle,-5*\wiggle) {$4$};
\node[scale=0.71] at (1*\wiggle,-5*\wiggle) {$9$};
\node[scale=0.71] at (2*\wiggle,-5*\wiggle) {$12$};
\node[scale=0.71] at (3*\wiggle,-5*\wiggle) {$8$};
\node[scale=0.71] at (4*\wiggle,-5*\wiggle) {$7$};
\node[scale=0.71] at (5*\wiggle,-5*\wiggle) {$3$};
\node[scale=0.71] at (6*\wiggle,-5*\wiggle) {$2$};
\node[scale=0.71] at (7*\wiggle,-5*\wiggle) {$11$};
\node[scale=0.71] at (8*\wiggle,-5*\wiggle) {$10$};
\node[scale=0.71] at (9*\wiggle,-5*\wiggle) {$6$};
\node[scale=0.71] at (10*\wiggle,-5*\wiggle) {$5$};
\node[scale=0.71] at (11*\wiggle,-5*\wiggle) {$1$};
\node[scale=0.71] at (12*\wiggle,-5*\wiggle) {$4$};

\node[scale=0.71] at (0*\wiggle,-6*\wiggle) {$3$};
\node[scale=0.71] at (1*\wiggle,-6*\wiggle) {$2$};
\node[scale=0.71] at (2*\wiggle,-6*\wiggle) {$11$};
\node[scale=0.71] at (3*\wiggle,-6*\wiggle) {$10$};
\node[scale=0.71] at (4*\wiggle,-6*\wiggle) {$6$};
\node[scale=0.71] at (5*\wiggle,-6*\wiggle) {$5$};
\node[scale=0.71] at (6*\wiggle,-6*\wiggle) {$1$};
\node[scale=0.71] at (7*\wiggle,-6*\wiggle) {$4$};
\node[scale=0.71] at (8*\wiggle,-6*\wiggle) {$9$};
\node[scale=0.71] at (9*\wiggle,-6*\wiggle) {$12$};
\node[scale=0.71] at (10*\wiggle,-6*\wiggle) {$8$};
\node[scale=0.71] at (11*\wiggle,-6*\wiggle) {$7$};
\node[scale=0.71] at (12*\wiggle,-6*\wiggle) {$3$};

\node[scale=0.71] at (0*\wiggle,-7*\wiggle) {$5$};
\node[scale=0.71] at (1*\wiggle,-7*\wiggle) {$1$};
\node[scale=0.71] at (2*\wiggle,-7*\wiggle) {$4$};
\node[scale=0.71] at (3*\wiggle,-7*\wiggle) {$9$};
\node[scale=0.71] at (4*\wiggle,-7*\wiggle) {$12$};
\node[scale=0.71] at (5*\wiggle,-7*\wiggle) {$8$};
\node[scale=0.71] at (6*\wiggle,-7*\wiggle) {$7$};
\node[scale=0.71] at (7*\wiggle,-7*\wiggle) {$3$};
\node[scale=0.71] at (8*\wiggle,-7*\wiggle) {$2$};
\node[scale=0.71] at (9*\wiggle,-7*\wiggle) {$11$};
\node[scale=0.71] at (10*\wiggle,-7*\wiggle) {$10$};
\node[scale=0.71] at (11*\wiggle,-7*\wiggle) {$6$};
\node[scale=0.71] at (12*\wiggle,-7*\wiggle) {$5$};

\node[scale=0.71] at (0*\wiggle,-8*\wiggle) {$8$};
\node[scale=0.71] at (1*\wiggle,-8*\wiggle) {$7$};
\node[scale=0.71] at (2*\wiggle,-8*\wiggle) {$3$};
\node[scale=0.71] at (3*\wiggle,-8*\wiggle) {$2$};
\node[scale=0.71] at (4*\wiggle,-8*\wiggle) {$11$};
\node[scale=0.71] at (5*\wiggle,-8*\wiggle) {$10$};
\node[scale=0.71] at (6*\wiggle,-8*\wiggle) {$6$};
\node[scale=0.71] at (7*\wiggle,-8*\wiggle) {$5$};
\node[scale=0.71] at (8*\wiggle,-8*\wiggle) {$1$};
\node[scale=0.71] at (9*\wiggle,-8*\wiggle) {$4$};
\node[scale=0.71] at (10*\wiggle,-8*\wiggle) {$9$};
\node[scale=0.71] at (11*\wiggle,-8*\wiggle) {$12$};
\node[scale=0.71] at (12*\wiggle,-8*\wiggle) {$8$};

\node[scale=0.71] at (0*\wiggle,-9*\wiggle) {$10$};
\node[scale=0.71] at (1*\wiggle,-9*\wiggle) {$6$};
\node[scale=0.71] at (2*\wiggle,-9*\wiggle) {$5$};
\node[scale=0.71] at (3*\wiggle,-9*\wiggle) {$1$};
\node[scale=0.71] at (4*\wiggle,-9*\wiggle) {$4$};
\node[scale=0.71] at (5*\wiggle,-9*\wiggle) {$9$};
\node[scale=0.71] at (6*\wiggle,-9*\wiggle) {$12$};
\node[scale=0.71] at (7*\wiggle,-9*\wiggle) {$8$};
\node[scale=0.71] at (8*\wiggle,-9*\wiggle) {$7$};
\node[scale=0.71] at (9*\wiggle,-9*\wiggle) {$3$};
\node[scale=0.71] at (10*\wiggle,-9*\wiggle) {$2$};
\node[scale=0.71] at (11*\wiggle,-9*\wiggle) {$11$};
\node[scale=0.71] at (12*\wiggle,-9*\wiggle) {$10$};

\node[scale=0.71] at (0*\wiggle,-10*\wiggle) {$9$};
\node[scale=0.71] at (1*\wiggle,-10*\wiggle) {$12$};
\node[scale=0.71] at (2*\wiggle,-10*\wiggle) {$8$};
\node[scale=0.71] at (3*\wiggle,-10*\wiggle) {$7$};
\node[scale=0.71] at (4*\wiggle,-10*\wiggle) {$3$};
\node[scale=0.71] at (5*\wiggle,-10*\wiggle) {$2$};
\node[scale=0.71] at (6*\wiggle,-10*\wiggle) {$11$};
\node[scale=0.71] at (7*\wiggle,-10*\wiggle) {$10$};
\node[scale=0.71] at (8*\wiggle,-10*\wiggle) {$6$};
\node[scale=0.71] at (9*\wiggle,-10*\wiggle) {$5$};
\node[scale=0.71] at (10*\wiggle,-10*\wiggle) {$1$};
\node[scale=0.71] at (11*\wiggle,-10*\wiggle) {$4$};
\node[scale=0.71] at (12*\wiggle,-10*\wiggle) {$9$};

\node[scale=0.71] at (0*\wiggle,-11*\wiggle) {$2$};
\node[scale=0.71] at (1*\wiggle,-11*\wiggle) {$11$};
\node[scale=0.71] at (2*\wiggle,-11*\wiggle) {$10$};
\node[scale=0.71] at (3*\wiggle,-11*\wiggle) {$6$};
\node[scale=0.71] at (4*\wiggle,-11*\wiggle) {$5$};
\node[scale=0.71] at (5*\wiggle,-11*\wiggle) {$1$};
\node[scale=0.71] at (6*\wiggle,-11*\wiggle) {$4$};
\node[scale=0.71] at (7*\wiggle,-11*\wiggle) {$9$};
\node[scale=0.71] at (8*\wiggle,-11*\wiggle) {$12$};
\node[scale=0.71] at (9*\wiggle,-11*\wiggle) {$8$};
\node[scale=0.71] at (10*\wiggle,-11*\wiggle) {$7$};
\node[scale=0.71] at (11*\wiggle,-11*\wiggle) {$3$};
\node[scale=0.71] at (12*\wiggle,-11*\wiggle) {$2$};

\node[scale=0.71] at (0*\wiggle,-12*\wiggle) {$1$};
\node[scale=0.71] at (1*\wiggle,-12*\wiggle) {$4$};
\node[scale=0.71] at (2*\wiggle,-12*\wiggle) {$9$};
\node[scale=0.71] at (3*\wiggle,-12*\wiggle) {$12$};
\node[scale=0.71] at (4*\wiggle,-12*\wiggle) {$8$};
\node[scale=0.71] at (5*\wiggle,-12*\wiggle) {$7$};
\node[scale=0.71] at (6*\wiggle,-12*\wiggle) {$3$};
\node[scale=0.71] at (7*\wiggle,-12*\wiggle) {$2$};
\node[scale=0.71] at (8*\wiggle,-12*\wiggle) {$11$};
\node[scale=0.71] at (9*\wiggle,-12*\wiggle) {$10$};
\node[scale=0.71] at (10*\wiggle,-12*\wiggle) {$6$};
\node[scale=0.71] at (11*\wiggle,-12*\wiggle) {$5$};
\node[scale=0.71] at (12*\wiggle,-12*\wiggle) {$1$};

\end{tikzpicture}
\end{equation*}
In particular, every row and column of the McKay quiver contains the vertex $1$, which corresponds to the fact that $ab$ has order equal to half of the order of $G$. By Theorem \ref{AuslanderForCoaction}, it follows that the Auslander map is an isomorphism in this case. \\
\indent The toroidal presentation of the quiver also shows that the invariant ring $A^H$ is generated by five elements, which correspond to paths in $\Lambda$ from vertex $1$ to itself of lengths $4,8,8,12,$ and $12$. In fact, one can show that the McKay quiver is identical to the McKay quiver of a $\frac{1}{12}(1,5)$ singularity, and so $A^H$ is a non-Gorenstein type $\mathbb{A}$ singularity (see Sections 5 and 6 of \cite{gl2c}).
\end{example}

\subsection{Regularity of the invariant ring}
In this section, we determine precisely when the invariant ring corresponding to a pair $(A,G)$ from Proposition \ref{prop:Classifying} is AS regular. \\
\indent Suppose that $G = \langle a,b \mid \mathcal{R} \rangle$ is as in Proposition \ref{prop:Classifying} so that, in particular, $a^2 = b^2$ must be a relation in $G$. Since $G$ is finite, it follows that $a$ and $b$ must have the same order, as must $ab$ and $ba$; call these orders $i$ and $m$, respectively, both of which are $\geqslant 2$ since $G$ is non-abelian. In particular,
\begin{align*}
1_G = (ab)^m (ba)^m = a^{4m}    
\end{align*}
so that $i \leqslant 4m$. It follows that $G$ must be a quotient of the group
\begin{align*}
\Gamma_m \coloneqq \langle a,b \mid a^2=b^2, a^{4m} = b^{4m} = (ab)^m = (ba)^m = 1_G \rangle.
\end{align*}

\begin{lem}
$|\Gamma_m| = 4m^2$.
\end{lem}
\begin{proof}
Letting $\omega$ be a primitive $2m$th root of unity, if we identify $a$ and $b$ with the matrices 
\begin{align*}
\begin{pmatrix}
0 & \omega \\ 1 & 0
\end{pmatrix}
\quad \text{and} \quad 
\begin{pmatrix}
0 & 1 \\ \omega & 0
\end{pmatrix},
\end{align*}
then it is easy to check that they satisfy the relations in $\Gamma_m$. This gives rise to a representation of some quotient of $\Gamma_m$, but in particular it tells us that $a$ has order $4m$. Noting that $\Gamma_m/\langle a^2 \rangle \cong D_m$, where $D_m$ is the dihedral group of order $2m$, we find that
\begin{align*}
\frac{|\Gamma_m|}{2m} = \frac{|\Gamma_m|}{|\langle a^2 \rangle|} = |D_m| = 2m,
\end{align*}
so that $|\Gamma_m| = 4m^2$.
\end{proof}

Using this, we can prove the main result of this section:

\begin{thm} \label{thm:InvariantRegular}
$A_{1_G}$ is AS regular if and only if $G = \Gamma_m$ for some $m$.
\end{thm}
\begin{proof}
With notation as above, if we assume that $ab$ has order $m$ then it follows that $G$ must be a quotient of $\Gamma_m$. Therefore the toroidal presentation of the associated McKay quiver has the following form:
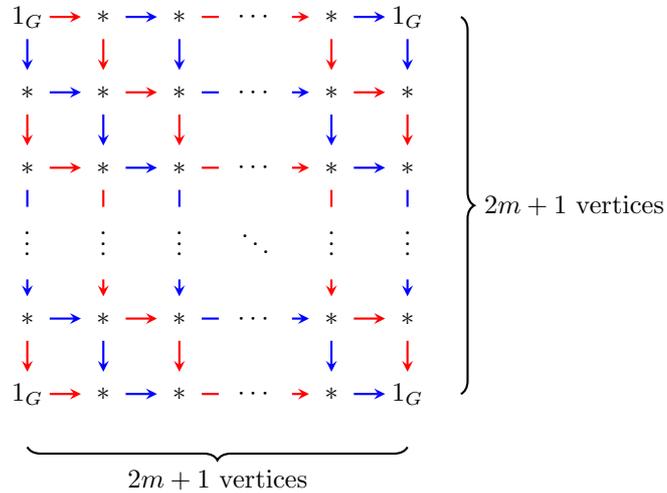
\begin{equation*}
\begin{tikzpicture}[thick,scale=1]
\def \wiggle {1};

\node at (-4.4,0) {}; %used to centre the quiver

\node at (0*\wiggle,0*\wiggle) {$1_G$};
\node at (1*\wiggle,0*\wiggle) {$*$};
\node at (2*\wiggle,0*\wiggle) {$*$};
\node at (3*\wiggle,0*\wiggle) {$\cdots$};
\node at (4*\wiggle,0*\wiggle) {$*$};
\node at (5*\wiggle,0*\wiggle) {$1_G$};

\node at (0*\wiggle,-1*\wiggle) {$*$};
\node at (1*\wiggle,-1*\wiggle) {$*$};
\node at (2*\wiggle,-1*\wiggle) {$*$};
\node at (3*\wiggle,-1*\wiggle) {$\cdots$};
\node at (4*\wiggle,-1*\wiggle) {$*$};
\node at (5*\wiggle,-1*\wiggle) {$*$};

\node at (0*\wiggle,-2*\wiggle) {$*$};
\node at (1*\wiggle,-2*\wiggle) {$*$};
\node at (2*\wiggle,-2*\wiggle) {$*$};
\node at (3*\wiggle,-2*\wiggle) {$\cdots$};
\node at (4*\wiggle,-2*\wiggle) {$*$};
\node at (5*\wiggle,-2*\wiggle) {$*$};

\node at (0*\wiggle,-3*\wiggle+0.1) {$\vdots$};
\node at (1*\wiggle,-3*\wiggle+0.1) {$\vdots$};
\node at (2*\wiggle,-3*\wiggle+0.1) {$\vdots$};
\node at (3*\wiggle,-3*\wiggle+0.1) {$\ddots$};
\node at (4*\wiggle,-3*\wiggle+0.1) {$\vdots$};
\node at (5*\wiggle,-3*\wiggle+0.1) {$\vdots$};

\node at (0*\wiggle,-4*\wiggle) {$*$};
\node at (1*\wiggle,-4*\wiggle) {$*$};
\node at (2*\wiggle,-4*\wiggle) {$*$};
\node at (3*\wiggle,-4*\wiggle) {$\cdots$};
\node at (4*\wiggle,-4*\wiggle) {$*$};
\node at (5*\wiggle,-4*\wiggle) {$*$};

\node at (0*\wiggle,-5*\wiggle) {$1_G$};
\node at (1*\wiggle,-5*\wiggle) {$*$};
\node at (2*\wiggle,-5*\wiggle) {$*$};
\node at (3*\wiggle,-5*\wiggle) {$\cdots$};
\node at (4*\wiggle,-5*\wiggle) {$*$};
\node at (5*\wiggle,-5*\wiggle) {$1_G$};

\node [minimum width=0.56cm,minimum height=0.56cm](1) at (0*\wiggle,0*\wiggle) {};
\node [minimum width=0.56cm,minimum height=0.56cm](2) at (1*\wiggle,0*\wiggle) {};
\node [minimum width=0.56cm,minimum height=0.56cm](3) at (2*\wiggle,0*\wiggle) {};
\node [minimum width=0.56cm,minimum height=0.56cm](4) at (3*\wiggle,0*\wiggle) {};
\node [minimum width=0.56cm,minimum height=0.56cm](5) at (4*\wiggle,0*\wiggle) {};
\node [minimum width=0.56cm,minimum height=0.56cm](6) at (5*\wiggle,0*\wiggle) {};

\node [minimum width=0.56cm,minimum height=0.56cm](8) at (0*\wiggle,-1*\wiggle) {};
\node [minimum width=0.56cm,minimum height=0.56cm](9) at (1*\wiggle,-1*\wiggle) {};
\node [minimum width=0.56cm,minimum height=0.56cm](10) at (2*\wiggle,-1*\wiggle) {};
\node [minimum width=0.56cm,minimum height=0.56cm](11) at (3*\wiggle,-1*\wiggle) {};
\node [minimum width=0.56cm,minimum height=0.56cm](12) at (4*\wiggle,-1*\wiggle) {};
\node [minimum width=0.56cm,minimum height=0.56cm](13) at (5*\wiggle,-1*\wiggle) {};

\node [minimum width=0.56cm,minimum height=0.56cm](15) at (0*\wiggle,-2*\wiggle) {};
\node [minimum width=0.56cm,minimum height=0.56cm](16) at (1*\wiggle,-2*\wiggle) {};
\node [minimum width=0.56cm,minimum height=0.56cm](17) at (2*\wiggle,-2*\wiggle) {};
\node [minimum width=0.56cm,minimum height=0.56cm](18) at (3*\wiggle,-2*\wiggle) {};
\node [minimum width=0.56cm,minimum height=0.56cm](19) at (4*\wiggle,-2*\wiggle) {};
\node [minimum width=0.56cm,minimum height=0.56cm](20) at (5*\wiggle,-2*\wiggle) {};

\node [minimum width=0.56cm,minimum height=0.56cm](22) at (0*\wiggle,-3*\wiggle) {};
\node [minimum width=0.56cm,minimum height=0.56cm](23) at (1*\wiggle,-3*\wiggle) {};
\node [minimum width=0.56cm,minimum height=0.56cm](24) at (2*\wiggle,-3*\wiggle) {};
\node [minimum width=0.56cm,minimum height=0.56cm](25) at (3*\wiggle,-3*\wiggle) {};
\node [minimum width=0.56cm,minimum height=0.56cm](26) at (4*\wiggle,-3*\wiggle) {};
\node [minimum width=0.56cm,minimum height=0.56cm](27) at (5*\wiggle,-3*\wiggle) {};

\node [minimum width=0.56cm,minimum height=0.56cm](29) at (0*\wiggle,-4*\wiggle) {};
\node [minimum width=0.56cm,minimum height=0.56cm](30) at (1*\wiggle,-4*\wiggle) {};
\node [minimum width=0.56cm,minimum height=0.56cm](31) at (2*\wiggle,-4*\wiggle) {};
\node [minimum width=0.56cm,minimum height=0.56cm](32) at (3*\wiggle,-4*\wiggle) {};
\node [minimum width=0.56cm,minimum height=0.56cm](33) at (4*\wiggle,-4*\wiggle) {};
\node [minimum width=0.56cm,minimum height=0.56cm](34) at (5*\wiggle,-4*\wiggle) {};

\node [minimum width=0.56cm,minimum height=0.56cm](36) at (0*\wiggle,-5*\wiggle) {};
\node [minimum width=0.56cm,minimum height=0.56cm](37) at (1*\wiggle,-5*\wiggle) {};
\node [minimum width=0.56cm,minimum height=0.56cm](38) at (2*\wiggle,-5*\wiggle) {};
\node [minimum width=0.56cm,minimum height=0.56cm](39) at (3*\wiggle,-5*\wiggle) {};
\node [minimum width=0.56cm,minimum height=0.56cm](40) at (4*\wiggle,-5*\wiggle) {};
\node [minimum width=0.56cm,minimum height=0.56cm](41) at (5*\wiggle,-5*\wiggle) {};

\draw[->,>=stealth,red] (1) -- (2);
\draw[->,>=stealth,blue] (2) -- (3);
\draw[->,>=stealth,blue] (5) -- (6);

\draw[->,>=stealth,blue] (8) -- (9);
\draw[->,>=stealth,red] (9) -- (10);
\draw[->,>=stealth,red] (12) -- (13);

\draw[->,>=stealth,red] (15) -- (16);
\draw[->,>=stealth,blue] (16) -- (17);
\draw[->,>=stealth,blue] (19) -- (20);

\draw[->,>=stealth,blue] (29) -- (30);
\draw[->,>=stealth,red] (30) -- (31);
\draw[->,>=stealth,red] (33) -- (34);

\draw[->,>=stealth,red] (36) -- (37);
\draw[->,>=stealth,blue] (37) -- (38);
\draw[->,>=stealth,blue] (40) -- (41);

\draw[->,>=stealth,blue] (1) -- (8);
\draw[->,>=stealth,red] (2) -- (9);
\draw[->,>=stealth,blue] (3) -- (10);
\draw[->,>=stealth,red] (5) -- (12);
\draw[->,>=stealth,blue] (6) -- (13);

\draw[->,>=stealth,red] (8) -- (15);
\draw[->,>=stealth,blue] (9) -- (16);
\draw[->,>=stealth,red] (10) -- (17);
\draw[->,>=stealth,blue] (12) -- (19);
\draw[->,>=stealth,red] (13) -- (20);

\draw[->,>=stealth,red] (29) -- (36);
\draw[->,>=stealth,blue] (30) -- (37);
\draw[->,>=stealth,red] (31) -- (38);
\draw[->,>=stealth,blue] (33) -- (40);
\draw[->,>=stealth,red] (34) -- (41);

\node [minimum width=0,minimum height=0](100) at (2.65*\wiggle,0*\wiggle) {};
\node [minimum width=0,minimum height=0](101) at (3.35*\wiggle,0*\wiggle) {};
\draw[red] (3) -- (100);
\draw[->,>=stealth,red] (101) -- (5);

\node [minimum width=0,minimum height=0](107) at (2.65*\wiggle,-1*\wiggle) {};
\node [minimum width=0,minimum height=0](108) at (3.35*\wiggle,-1*\wiggle) {};
\draw[blue] (10) -- (107);
\draw[->,>=stealth,blue] (108) -- (12);

\node [minimum width=0,minimum height=0](114) at (2.65*\wiggle,-2*\wiggle) {};
\node [minimum width=0,minimum height=0](115) at (3.35*\wiggle,-2*\wiggle) {};
\draw[red] (17) -- (114);
\draw[->,>=stealth,red] (115) -- (19);

\node [minimum width=0,minimum height=0](128) at (2.65*\wiggle,-4*\wiggle) {};
\node [minimum width=0,minimum height=0](129) at (3.35*\wiggle,-4*\wiggle) {};
\draw[blue] (31) -- (128);
\draw[->,>=stealth,blue] (129) -- (33);

\node [minimum width=0,minimum height=0](135) at (2.65*\wiggle,-5*\wiggle) {};
\node [minimum width=0,minimum height=0](136) at (3.35*\wiggle,-5*\wiggle) {};
\draw[red] (38) -- (135);
\draw[->,>=stealth,red] (136) -- (40);

\node [minimum width=0,minimum height=0](200) at (0*\wiggle,-2.65*\wiggle) {};
\node [minimum width=0,minimum height=0](207) at (0*\wiggle,-3.35*\wiggle) {};
\draw[blue] (15) -- (200);
\draw[->,>=stealth,blue] (207) -- (29);

\node [minimum width=0,minimum height=0](201) at (1*\wiggle,-2.65*\wiggle) {};
\node [minimum width=0,minimum height=0](208) at (1*\wiggle,-3.35*\wiggle) {};
\draw[red] (16) -- (201);
\draw[->,>=stealth,red] (208) -- (30);

\node [minimum width=0,minimum height=0](202) at (2*\wiggle,-2.65*\wiggle) {};
\node [minimum width=0,minimum height=0](209) at (2*\wiggle,-3.35*\wiggle) {};
\draw[blue] (17) -- (202);
\draw[->,>=stealth,blue] (209) -- (31);

\node [minimum width=0,minimum height=0](204) at (4*\wiggle,-2.65*\wiggle) {};
\node [minimum width=0,minimum height=0](211) at (4*\wiggle,-3.35*\wiggle) {};
\draw[red] (19) -- (204);
\draw[->,>=stealth,red] (211) -- (33);

\node [minimum width=0,minimum height=0](205) at (5*\wiggle,-2.65*\wiggle) {};
\node [minimum width=0,minimum height=0](212) at (5*\wiggle,-3.35*\wiggle) {};
\draw[blue] (20) -- (205);
\draw[->,>=stealth,blue] (212) -- (34);

\draw [decorate, 
    decoration = {brace,
        raise=20pt,
        amplitude=5pt}] (5*\wiggle,0*\wiggle) --  (5*\wiggle,-5*\wiggle) node[pos=0.5,right=25pt,black]{$2m+1$ vertices};
        
\draw [decorate, 
    decoration = {brace, mirror,
        raise=20pt,
        amplitude=5pt}] (0*\wiggle,-5*\wiggle) --  (5*\wiggle,-5*\wiggle) node[pos=0.5,below=25pt,black]{$2m+1$ vertices};

\end{tikzpicture}
%\label{RegularMcKayQuiv}
\end{equation*}
In this diagram, the asterisks represent arbitrary elements of $G$, which may include the identity. \\
\indent In light of our conventions from Section \ref{sec:Quivers}, the red arrows are decorated by $u$ and the blue arrows by $v$, the sequence of vertices in the top row is $1_G$, $a^{-1}$, $b^{-1} a^{-1}$, $a^{-1} b^{-1} a^{-1}, \dots, 1_G$, and there are $2m+1$ vertices since $ab$ has order $m$. A similar analysis applies to the left hand column. In particular, none of the asterisks in the first and last row, or the leftmost and rightmost columns, correspond to $1_G$. \\
\indent Now, by Theorem \ref{thm:QuiverAlgebra}, we have $A_{1_G} \cong e_{1} \Lambda e_{1}$. From the above presentation of $Q$, we see that there is a horizontal path and a vertical path from $1_G$ to itself, both of length $2m$, and neither of these can be factored as a product of shorter paths from $1_G$ to itself. Therefore these elements can be chosen to be elements in the generating set for $e_{1} \Lambda e_{1}$; call the horizontal path $x$ and the vertical path $y$. In $A_{1_G}$, these elements correspond to $(uv)^m$ and $(vu)^m$. \\
\indent First suppose that $G = \Gamma_m$ for some $m$. If we ignore the vertices on the right hand side and the bottom row of the McKay quiver (which are already accounted for by their identification with the left hand side and top row, respectively) there are exactly $2m \times 2m = 4m^2$ vertices remaining. Moreover, each vertex is labelled by a group element and each group element must appear at least once among these vertices. Since $|\Gamma_m| = 4m^2$, each group element appears exactly once. In particular, the identity appears only at the corners of the toroidal presentation. Therefore every path from $1_G$ to itself in the quiver can be obtained as a composition of the paths $x$ and $y$, so that $x$ and $y$ generate $e_{1} \Lambda e_{1}$. It is straightforward to check that they commute and satisfy no other relations, so that $A_{1_G} \cong e_{1} \Lambda e_{1}$ is a polynomial ring in two variables. In particular, it is AS regular.
\\
\indent Now suppose that $G \neq \Gamma_m$. In particular, it is a proper quotient $\Gamma_m/N$ of $\Gamma_m$, where $N$ is a normal subgroup of $\Gamma_m$; we claim that this forces $i$, the order of the image of $a$ in $\Gamma_m/N$, to be strictly less than $4m$. Indeed, if $i = 4m$ then
\begin{align*}
\frac{|\Gamma_m/N|}{2m} = \left| \frac{\Gamma_m/N}{\langle a^2 N \rangle/N} \right| = \left| \frac{\Gamma_m}{\langle a^2 \rangle} \right| = | D_m | = 2m,
\end{align*}
where $D_m$ is the dihedral group of order $2m$. Consequently, $\Gamma_m/N$ has order $4m^2$, contradicting the fact that it is a proper quotient of $\Gamma_m$. It follows that we must have $i < 4m$. Now, $u^i \in A_{1_G}$, and it is certainly not a linear combination of $(uv)^m$ and $(vu)^m$. Also, it cannot be a quadratic (or higher) expression in $(uv)^m$ and $(vu)^m$, since such an expression has degree (at least) $4m$, which is greater than $i$. We note that $u^i$ need not necessarily be a generator itself, but in this case an additional generator is required to be able to obtain $u^i$. In either case, $A_{1_G}$ requires at least three generators. However, if $A_{1_G}$ were AS regular, then it would be 2-dimensional, and hence have two generators; therefore it cannot be AS regular.
\end{proof}

\begin{rem}
%When $A_{1_G}$ is not AS regular, our testing suggests that, at least for groups of small order, it is usually the case that we can choose $u^i$ to be a generator, for some $i$. 
The smallest group (with its implicit $G$-grading on $\Bbbk\langle u,v \rangle/\langle u^2-v^2 \rangle$) for which it is not possible to take $u^i$ as a generator for the invariant ring has order $48$, and presentation
\begin{align*}
G = \langle a,b \mid a^2=b^2, a^{16} = (a^7b)^3 = 1_G \rangle.
\end{align*}
In this case, the smallest power of $u$ that is an invariant is $u^{16}$ and $A_{1_G}$ is generated by $(uv)^{24}$, $(vu)^{24}$, $u^3(vu)^2v$, and $u^2(vu)^3$; here, $u^{16} = (u^3(vu)^2v)(u^2(vu)^3)$.
\end{rem}

\bibliographystyle{amsalpha}
\bibliography{bibliography}

\end{document}